\documentclass{elsarticle}
\usepackage[utf8]{inputenc}
\usepackage{amssymb}
\usepackage[english]{babel}
\usepackage{mathrsfs}
\usepackage{euscript}
\usepackage{amsmath}
\usepackage{amssymb}
\usepackage{amsthm}
\usepackage{amsfonts}
\usepackage{amsbsy}
\usepackage{enumitem}
\usepackage{natbib}
\renewcommand{\le}{\leqslant}
\renewcommand{\ge}{\geqslant}

\newcommand{\ddd}{\,\text{d}}
\newtheorem{theorem}{Theorem}[section]
\newtheorem{definition}[theorem]{Definition}
\newtheorem{lemma}[theorem]{Lemma}
\newtheorem{remark}[theorem]{Remark}
\newtheorem{corollary}[theorem]{Corollary}
\theoremstyle{definition}
\newtheorem{example}[theorem]{Example}
\newtheorem{problem}[theorem]{Problem}
\newcommand{\ind}{\text{ind}}
\newcommand{\supp}{\text{supp}\,}
\newcommand{\intr}{\text{int}\,}
\setlist[enumerate,1]{label=(\roman*)}
\numberwithin{equation}{section}

\begin{document}
	\title{Subnormal and completely hyperexpansive completion problem for weighted shifts on directed trees\tnoteref{t1}}
	\tnotetext[t1]{The author of the publication received an incentive scholarship from the 
	funds of the program Excellence Initiative - Research University at the Jagiellonian 
	University in Kraków.}
	\author[IMUJ,DSUJ]{Michał Buchała}
	\affiliation[IMUJ]{organization={Institute of Mathematics, Jagiellonian University},
		addressline={Łojasiewicza~6},
		postcode={PL-30348},
		city={Krakow},
		country={Poland}}
	\affiliation[DSUJ]{organization={Doctoral School of Exact and Natural Sciences, Jagiellonian University},
		addressline={Łojasiewicza~11},
		postcode={PL-30348},
		city={Krakow},
		country={Poland}}
%	\keywords{}

	\begin{abstract}
		For a given directed tree and weights attached to a subtree, the completion problem is to determine if these weights may be completed in a way to obtain a bounded weighted shift on the whole tree, which further satisfies additional conditions. In this paper we consider subnormal and completely hyperexpansive completion problem for weighted shifts on directed trees with one branching point. We develop new results on backward extensions of truncated moment sequences and, exploiting these results, we obtain a characterization of existence of such a completion
	\end{abstract}

   \maketitle

	\section{Introduction}

\label{SecIntroduction}
Classical weighted shifts form an important class of operators, which is a rich source of examples in operator theory. In \cite{Sta1} the author investigated the subnormality of weighted shifts and gave the construction of a minimal normal extension of a subnormal weighted shift. In \cite{GelWal1} there appeared another characterization of subnormal weighted shifts in terms of moment sequences (the characterization was mentioned also in \cite{Hal2} as done formerly by Berger). In \cite{Ath1} the author investigated complete hyperexpansivity of weighted shifts on the occasion of introducing such a class of operators; the obtained characterization of completely hyperexpansive weighted shifts involves L\'evy-Khinchin representation of certain sequences. \par 
In \cite{JJS1} weighted shifts on directed trees have been introduced. It is a more general class of operators containing classical weighted shifts as well as their adjoints (cf. \cite[Remark 3.4.2]{JJS1}). In the same paper there appeared criteria on subnormality and complete hyperexpansivity of weighted shifts on trees with one branching point (cf. \cite[Theorems 6.2.1 and 7.2.1]{JJS1}). These criteria, as for classical weighted shifts, translate notions of subnormality and complete hyperexpansivity into certain moment sequences, but this time negative moments are also involved. \par
In \cite{Sta1} Stampfli initiated research on subnormal completions of weighted shifts. Thanks to his construction of minimal normal extension of a subnormal weighted shift, he obtained that for $ 0<\lambda_1<\lambda_2<\lambda_3 $ there exists a subnormal weighted shift with weights starting from $ (\lambda_{1},\lambda_{2}, \lambda_{3}) $. Moreover, he showed that such a condition does not suffice if we take four weights. Nevertheless, in his paper there appears a quite simple condition for a sequence $ 0<\lambda_1<\lambda_2<\lambda_3<\lambda_4 $ to have a subnormal completion, namely
\begin{equation*}
\lambda_{4}^{2} \ge \lambda_{3}^{2} + \frac{\lambda_{1}^{2}}{\lambda_{3}^{2}}\frac{(\lambda_{3}^{2}-\lambda_{2}^{2})^{2}}{\lambda_{2}^{2}-\lambda_{1}^{2}}.
\end{equation*}
In \cite{CF1} the authors recovered and generalized the results of \cite{Sta1} by studying truncated moment sequences and recursively generated completions. In \cite{JabStoJungKwak1} the completely hyperexpansive completion problem for weighted shifts was investigated -- the authors used similar approach as in \cite{CF1}. In \cite{Li1} and \cite{CLY1} the authors studied subnormal completion problem for 2-variable weighted shifts. In \cite{K1} completions of  $ d $-variable weighted shifts were studied.\par
In \cite{EJSY1} there was posed a subnormal completion problem for weighted shifts on directed trees with one branching point, and several general results were proved (concerning e.g. 1-generation completions and flat completions). In \cite[Theorem 4.1]{EJSY2} the solution to the 2-generation completion problem was obtained for directed trees with trunk of length 1.\par
The aim of the present paper is to give a full solution to $ p $-generation subnormal and completely hyperexpansive completion problem for weighted shifts on directed trees with one branching point. The paper is organized as follows. In Section \ref{SecMomentsOnInfiniteInterval} we study truncated moment sequences on $ (0,\infty) $ and $ (0,1] $ using methods from \cite[Chapters III and IV]{KN1}. The main result of this section is Theorem \ref{ThmStrictlyPositiveSequencesInfiniteInterval}, which characterizes strict positivity in terms of index of a given sequence. In Section \ref{SecBackwardExtensionsInfiniteInterval} we study backward extensions of truncated moment sequences on $ (0,\infty) $ and $ (0,1] $. The key result is Theorem \ref{ThmStrictlyPositiveBackwardExtensionsInfiniteInterval}, which gives us the simple condition for strict positivity of backward extension. These all allow us to obtain crucial characterizations of truncated moment sequences in terms of extremal values of the integral of reciprocal function (see Theorems \ref{ThmConditionsOnMomentsOfMinimalMeasureInfiniteInterval}-\ref{ThmExistenceOfBackwardExtensionWithPrescribedNumberOfAtomsHalfOpenInterval}). In Theorem \ref{ThmSubnormalCompletion}, using the theory developed in Section \ref{SecBackwardExtensionsInfiniteInterval}, we give a solution (in the kind of \cite[Theorem 4.1]{EJSY2}) to the subnormal completion problem for directed tress with a finite trunk. The next part of Section \ref{SecSubnormalCompletions} is devoted to study trees with infinite trunk -- exploiting compactness of certain sets, we link these two cases. Theorem \ref{ThmCompletelyHyperexpansiveCompletion} is an analogue of Theorem \ref{ThmSubnormalCompletion} for the completely hyperexpansive completion problem. 
	\section{Preliminaries}
\label{SecPreliminaries}
By $ \mathbb{N} $, $ \mathbb{Z} $, $ \mathbb{Q} $, $ \mathbb{R} $ we denote the set of non-negative integers, integers, rational numbers, and real numbers, respectively. We also use the symbol $ \overline{\mathbb{N}} $ for the set $ \mathbb{N}\cup \{\infty\} $ and $ \overline{\mathbb{R}} $ for the set $ \mathbb{R}\cup \{-\infty,\infty \} $. For $ p\in\mathbb{R} $ and $ A\subset \overline{\mathbb{R}} $ we denote $ A_p = \{x\in A\!: x\ge p \} $. For $ x\in\mathbb{R} $ we denote
\begin{equation*}
    \lceil x \rceil = \min\{k\in \mathbb{Z}\!: x\le k \}.
\end{equation*}
For a subset $ A\subset \mathbb{R} $ we denote by $ \chi_{A} $ the characteristic function of $ A $. For a sequence $ (s_{i})_{i\in I} $ ($ I\subset \mathbb{N} $) of real numbers the notation $ (s_{i})_{i\in I} \subset A $ means that all terms of the sequence $ (s_{i})_{i\in I} $ are elements of the set $ A \subset \mathbb{R} $. If $ \mathbb{K}\in\{\mathbb{R}, \mathbb{C}\} $ and $ n\in\mathbb{N} $, then $ \mathbb{K}_{n}[x] $ stands for the space of polynomials of degree at most $ n $ with coefficients in $ \mathbb{K} $. If $ X $ is a topological space, then $ \mathcal{B}(X) $ stands for the $ \sigma $-algebra of Borel subsets of $ X $. If $ \mu $ is a positive regular measure on $ \mathcal{B}(X) $, then the closed support of $ \mu $ is denoted by $ \supp\mu $.\par
A sequence $ (s_{k})_{k=0}^{\infty}\subset [0,\infty) $ is called a \textit{Stjelties moment sequence}, if there exists a measure $ \mu\!:\mathcal{B}([0,\infty))\to [0,\infty) $ such that
\begin{equation}
        \int_{[0,\infty)} t^{n}\ddd\mu(t) = s_{k}, \qquad k\in\mathbb{N}.
\end{equation}
Recall that a sequence $ (a_{k})_{k=0}^{\infty} \subset \mathbb{R} $ is \textit{completely alternating} if 
\begin{equation*}
    \sum_{k=0}^{n} (-1)^{k} \binom{n}{k} a_{k+m} \le 0, \qquad m\in\mathbb{N},\ n\in\mathbb{N}_{1}.
\end{equation*}
Let $ H $ be a complex Hilbert space. By $ \mathbf{B}(H) $ we denote the $ C^{\ast} $-algebra of linear and bounded operators on $ H $. An operator $ T\in \mathbf{B}(H) $ is called \textit{normal} if $ T $ commutes with its adjoint. An operator $ T\in\mathbf{B}(H) $ is called \textit{subnormal} if there exists a Hilbert space $ K $ containing $ H $ (in sense of isometric embeddings) and a normal operator $ S\in\mathbf{B}(K) $ such that $ S|_{H} = T $.
We say that $ T\in\mathbf{B}(H) $ is \textit{completely hyperexpansive} if
\begin{equation}
    \label{FormCompletelyHyperExpansiveOperator}
    \sum_{k=0}^{n}(-1)^{k}\binom{n}{k} T^{\ast k}T^{k} \le 0, \qquad n\in\mathbb{N}_{1}.
\end{equation}
In view of \cite[Remark 2]{Ath1} we see that $ T $ is completely hyperexpansive if and only if for every $ f\in H $ the sequence $ (\lVert T^{n}f\rVert^{2})_{n=0}^{\infty} $ is completely alternating.\par
Let us recall the definition of weighted shift on a directed tree (see \cite{JJS1}). For a directed tree $ \mathcal{T} = (V,E) $ and a family $ \pmb{\lambda} = (\lambda_{v})_{v\in V\setminus\{\mathtt{root}\}}\subset \mathbb{C} $ satisfying the following condition:
\begin{equation}
    \label{FormBoundedShift}
    \sup_{v\in V} \sum_{\substack{u\in V\\(v,u)\in E}} \lvert \lambda_{u}\rvert^{2} < \infty,
\end{equation}
we define an operator $ S_{\pmb{\lambda}}\in \mathbf{B}(\ell^{2}(V)) $, called \textit{weighted shift} on $ \mathcal{T} $ with weights $ \pmb{\lambda} $, by the formula
\begin{equation*}
    S_{\pmb{\lambda}} e_{v} = \sum_{\substack{u\in V\\ (v,u)\in E}} \lambda_{u}e_{u}, \qquad v\in V.
\end{equation*}
Note that the classical weighted shifts are also weighted shifts in this more general setting: a suitable trees are simply $ \mathbb{Z} $ and $ \mathbb{N} $. \par In this paper we are considering weighted shifts on directed trees with one branching point, that is, trees $ \mathcal{T}_{\eta,\kappa} = (V_{\eta, \kappa}, E_{\eta,\kappa}) $, where $ \eta,\kappa\in \overline{\mathbb{N}} $, $ \eta\ge 1 $, and
\begin{align*}
V_{\eta,\kappa} &= \{-k\!: k \in \mathbb{N}\cap[0,\kappa] \} \sqcup \{(i,j)\!:j\in\mathbb{N}_{1}, i\in \mathbb{N}\cap[1,\eta] \},\\
E_{\kappa} &= \{(-k,-k+1)\!: k\in \mathbb{N}\cap[1,\kappa] \}\sqcup \{(0,(i,1))\!: i\in\mathbb{N}\cap [1,\eta]\},\\
E_{\eta,\kappa} &= E_{\kappa} \sqcup \{((i,j),(i,j+1))\!: i\in\mathbb{N}\cap[1,\eta],j\in\mathbb{N}_{1} \}.
\end{align*}
Given $ \eta,\kappa\in\overline{\mathbb{N}} $, $ \eta\ge 2 $, and a weighted shift $ S_{\pmb{\lambda}}\in \mathbf{B}(\ell^{2}(V_{\eta,\kappa})) $ we say that $ S_{\pmb{\lambda}} $ is $ r $-generation flat if $ \lambda_{i,j} = \lambda_{1,j} $ for $ j\in\mathbb{N}_{r} $ and $ i\in\mathbb{N}\cap [1,\eta] $.\par
If $ \eta,\kappa\in \overline{\mathbb{N}} $, $ \eta\ge 2 $, $ p\in\mathbb{N}_1 $, and $ \pmb{\lambda} = \{\lambda_{-k}\}_{k=0}^{\kappa-1}\cup \{ \lambda_{i,1} \}_{i,j=1}^{\eta,p}\subset \mathbb{C} $, we say that $ \pmb{\lambda} $ has a completion if there exists a weighted shift $ S_{\pmb{\lambda'}}\in \mathbf{B}(\ell^{2}(V_{\eta,\kappa})) $ such that
\begin{align*}
    \lambda_{-k}' &= \lambda_{-k}, \qquad k\in\mathbb{N}\cap[0,\kappa-1]\\
    \lambda_{i,j}' &= \lambda_{i,j}, \qquad i\in\mathbb{N}\cap[1,\eta], \ j\in\mathbb{N}\cap[1,p];
\end{align*}
every such operator will be called a \textit{completion} of $ \pmb{\lambda} $.
	\section{Moment problem on $ (0,\infty) $ and $ (0,1] $}
\label{SecMomentsOnInfiniteInterval}
In \cite{CF1} the authors gave a solution to the truncated Stieltjes and Hausdorff moment problems. In this section we investigate the special cases of these problems, namely the moment problems on $ (0,\infty) $ and $ (0,1] $. We use a different approach, which is similar to what Krein and Nudel'man presented in \cite[Chapter III]{KN1} for truncated moment problem on compact interval; actually, it turns out that the theory of moments on compact interval contains the theory of moments on $ (0,\infty) $ and $ (0,1] $. This approach seems to be more suitable for studying backward extensions of truncated moment sequences in Section \ref{SecBackwardExtensionsInfiniteInterval}. \par
\begin{definition}
    \label{DefMomentSequenceInfiniteInterval}
    Let $ n\in\mathbb{N} $ and $ \mathbf{s} = (s_0,\ldots,s_n) \subset \mathbb{R} $. Assume $ I \subset \mathbb{R} $ is an interval. We say that $ \mathbf{s} $ is a \textit{moment sequence on} $ I $ if there exists a
    measure $ \mu\!: \mathcal{B}(I)\to [0,\infty) $ satisfying the following conditions:
    \begin{enumerate}
        \item there exists an interval $ [a,b]\subset I $ ($ a<b $) such that $ \supp\mu \subset [a,b] $,
        \item\ $ \int_{I} t^{k}\ddd\mu(t) = s_k, $ $ k\in\mathbb{N}\cap [0,n] $.
    \end{enumerate}
    If such a measure is unique, we call $ \mathbf{s} $ \textit{determinate on} $ I $; otherwise, we call a moment sequence $ \mathbf{s} $ \textit{indeterminate on} $ I $.
\end{definition}
If $ \mathbf{s} = (s_0,\ldots,s_n) \subset \mathbb{R} $, $ n\in\mathbb{N} $, is a moment sequence on an interval $ I\subset \mathbb{R} $, then by $ \mathcal{M}_{I}(\mathbf{s}) $ we denote the set of all measures satisfying Definition \ref{DefMomentSequenceInfiniteInterval}; for the further use we make the abbreviations: $ \mathcal{M}_{\infty}(\mathbf{s}) := \mathcal{M}_{(0,\infty)}(\mathbf{s}) $ and $ \mathcal{M}_{1}(\mathbf{s}) := \mathcal{M}_{(0,1]}(\mathbf{s}) $.\par
For $ \mathbf{s} = (s_0,\ldots,s_n) \subset \mathbb{R} $, $ n\in\mathbb{N} $,
define a linear functional $ \sigma\!: \mathbb{R}_n[x]\to \mathbb{R} $ by the
formula
\[ \sigma(x^{k}) = s_k, \quad k\in\mathbb{N}\cap [0,n]. \]
\begin{definition}
    \label{DefPositiveSequenceInfiniteInterval}
    Let $ n\in\mathbb{N} $ and $ \mathbf{s} = (s_0,\ldots,s_n) \subset \mathbb{R} $. Assume $ I \subset \mathbb{R} $ is an interval. A sequence $ \mathbf{s} $ is \textit{positive on} $ I $ if there exists an interval $ [a,b]\subset I $ ($ a<b $) such that for every polynomial $ P\in\mathbb{R}_n[x] $ satisfying $ P(x) \ge 0 $, $ x\in[a,b] $, we have $ \sigma(P)\ge 0 $. A positive sequence $ \mathbf{s} $ is
    \textit{strictly positive on} $ I $ if there exists an interval $ [a,b]\subset I $ ($ a<b $) such that for every nonzero polynomial $ P\in\mathbb{R}_n[x] $ satisfying $ P(x) \ge 0 $, $ x\in[a,b] $, we have $ \sigma(P)>0 $. If $ \mathbf{s} $ is positive on $ I $, but not strictly positive, we call it \textit{singularly positive on} $ I $.
\end{definition}
The first theorem reveals the connection between moment sequences and positive sequences (cf. \cite[Theorem III.1.1]{KN1}).
\begin{theorem}
    \label{ThmEquivalentConditionForMomentSequenceInfiniteInterval}
    Let $ n\in\mathbb{N} $ and $ \mathbf{s} = (s_0,\ldots,s_n) \subset \mathbb{R} $. Assume $ I\subset \mathbb{R} $ is an interval. Then $ \mathbf{s} $ is a moment sequence on $ I $ if and only if $ \mathbf{s} $ is positive on $ I $.
\end{theorem}
\begin{proof}
    For the proof of the 'if' part suppose $
    \mu\!: \mathcal{B}(I)\to [0,\infty) $ satisfies Definition \ref{DefMomentSequenceInfiniteInterval}. Let $ [a,b]\subset \mathbb{R} $ ($ a<b $) be such that $ \supp\mu\subset [a,b] $. Then for
    \[ 
        \nu\!: \mathcal{B}([a,b]) \ni A\longmapsto \mu(A) \in [0,\infty)
    \]
    we have $ \nu\in \mathcal{M}_{[a,b]}(\mathbf{s}) $. Hence, by \cite[Theorem III.1.1]{KN1}, $ \mathbf{s} $ is positive on $ [a,b] $, which implies that $ \mathbf{s} $ is positive on $ I $.\\
    Conversely, by Definition \ref{DefPositiveSequenceInfiniteInterval}, $ \mathbf{s} $ is positive on some interval
    $ [a,b]\subset I $ ($ a<b $). By \cite[Theorem III.1.1]{KN1}, $ \mathbf{s} $ is a moment sequence on $ [a,b] $. If $
    \mu\in\mathcal{M}_{[a,b]}(\mathbf{s}) $, then\[ \nu\!: \mathcal{B}(I) \ni
    A\longmapsto \mu(A\cap [a,b]) \in [0,\infty) \]
    satisfies Definition \ref{DefMomentSequenceInfiniteInterval}.
\end{proof}
Note that by \cite[Theorem III.4.1]{KN1} strictly positive moment sequences are always indeterminate; the
converse will be proved later in Theorem \ref{ThmStrictlyPositiveSequencesInfiniteInterval}.\par

\begin{remark}
    Let $ I \subset \mathbb{R} $ be an interval. If $ \mathbf{s} = (s_0,\ldots,s_n) \subset \mathbb{R} $, $ n\in\mathbb{N} $, is positive on $ I $, then
    \[ \mathcal{M}_{I}(\mathbf{s}) = \bigcup \{\mathcal{M}_{[a,b]}(\mathbf{s})\!: \mathbf{s} \text{ is positive on } [a,b] \subset I \} \]
    with obvious identification
    \begin{align*}
    \mathcal{M}_{[a,b]}(\mathbf{s})\ni \mu \longmapsto \mu(\,\cdot\,\cap [a,b]) \in \mathcal{M}_{I}(\mathbf{s}), \qquad a<b.
    \end{align*}
\end{remark}
Assume $ I\subset \mathbb{R} $ is an interval and $ \mathbf{s} = (s_{0},\ldots,s_{n})\subset \mathbb{R} $ is a moment sequence on $ I $. Following \cite[III.\S 3]{KS1} we define the index of a measure $ \mu\in \mathcal{M}_{I}(\mathbf{s}) $ as follows: if $ \supp\mu $ is an infinite set, then we set $ \ind_{I}(\mathbf{s}) = \infty $, otherwise we set
\begin{equation*}
    \ind_{I}(\mu) = \sum_{c\in I} \chi_{\supp\mu}(c)\epsilon_{I}(c),
\end{equation*}
where $ \epsilon_{I}\!: I\to \{\frac{1}{2},1\} $ is defined by:
\begin{equation*}
    \epsilon_{I}(c) = \begin{cases}
        1, & \text{if } c\in \intr I\\
        \frac{1}{2}, & \text{otherwise}
    \end{cases}
    \qquad c\in I.
\end{equation*}
\begin{remark}
    In \cite[Section III.\S 4.1]{KS1} the authors used equivalent definition of index with the function $ \epsilon_{I}\!: I\to \{1,2\} $ given by:
    \begin{equation*}
        \epsilon_{I}(c) = \begin{cases}
            2, & \text{if } c\in \intr I\\
            1, & \text{otherwise}
        \end{cases}
        \qquad c\in I.
    \end{equation*}
\end{remark}
We define also the index of $
\mathbf{s} $ as
\[ \ind_{I}(\mathbf{s}) = \min\{\ind_{I}(\mu)\!:
\mu\in\mathcal{M}_{I}(\mathbf{s}) \}. \]
In view of \cite[Theorem 1]{R1}, $ \ind_{I}(\mathbf{s}) $ is always finite. As before, for the convenience we abbreviate: $ \ind_{\infty} := \ind_{(0,\infty)} $ and $ \ind_{1} := \ind_{(0,1]} $. \par
In \cite[Chapter III]{KN1} the authors was considering the case $ I = [a,b] $. In \cite[Theorem III.4.1]{KN1} there was proved that the sequence $ \mathbf{s} $ is singularly positive on $ [a,b] $ if and only if $ \ind_{[a,b]}(\mathbf{s}) \le \frac{n}{2} $. Moreover (cf. \cite[Theorem III.5.1]{KN1}), if $ \mathbf{s} $ is strictly positive on $ [a,b] $, then there are exactly two measures (called \textit{principal measures}) in $ \mathcal{M}_{I}(\mathbf{s}) $ of index $ \frac{n+1}{2} $; the principal measure with atom at $ b $ is called \textit{upper principal} and the other is called \textit{lower principal}. In the subsequent part of this section we present counterpart of this theory for the intervals $ (0,\infty) $ and $ (0,1] $.\par
The next result gives the upper bound on the index of a positive sequence.
\begin{lemma}
    \label{LemIndexEstimationInfiniteInterval}
    Let $ n\in\mathbb{N} $ and $ \mathbf{s} = (s_0,s_1,\ldots,s_n) \subset [0,\infty) $.
    \begin{enumerate}
        \item If $ \mathbf{s} $ is positive on $ (0,\infty) $, then $ \ind_{\infty}(\mathbf{s}) \le \left\lceil \frac{n+1}{2}\right\rceil $.
        \item If $ \mathbf{s} $ is positive on $ (0,1] $, then $ \ind_{1}(\mathbf{s}) \le \frac{n+1}{2} $.
    \end{enumerate}
\end{lemma}
\begin{proof}
    We prove (i). First, assume $ \mathbf{s} $ is strictly positive on $ (0,\infty) $. Let $ [a,b]\subset (0,\infty) $ be an interval such that $ \mathbf{s} $ is strictly positive on $
    [a,b] $. Assume $ n = 2m+1 $, $ m\in\mathbb{N} $, and let $
    \mu\in\mathcal{M}_{[a,b]}(\mathbf{s}) $ be a lower principal measure. Then, in view of \cite[Section III.\S 4.1]{KN1},
    \[ \ind_{[a,b]}(\mu) = \ind_{\infty}(\mu) = m+1 = \frac{n+1}{2}, \]
    so $ \ind_{\infty}(\mathbf{s}) \le \frac{n+1}{2} = \left\lceil
    \frac{n+1}{2}\right\rceil $. Now, suppose $ n = 2m $, $ m \in \mathbb{N} $, and let $
    \mu\in\mathcal{M}_{[a,b]}(\mathbf{s}) $ be any principal measure. Then, again by \cite[Section III.\S 4.1]{KN1},
    \[ \ind_{[a,b]}(\mu) = \frac{n+1}{2} = m+\frac{1}{2} \]
    and
    \[ \ind_{\infty}(\mu) = m+1 = \left\lceil \frac{n+1}{2}\right\rceil, \]
    so $ \ind_{\infty}(\mathbf{s}) \le \left\lceil \frac{n+1}{2}\right\rceil
    $. Next, assume $ \mathbf{s} $ is singularly positive on $ (0,\infty) $, that is, for every interval $ [a,b]\subset (0,\infty) $ if $ \mathbf{s} $~is
    positive on $ [a,b] $, then $ \mathbf{s} $ is singularly positive on $ [a,b] $. Suppose to the contrary that $
    \ind_{\infty}(\mathbf{s}) > \left\lceil \frac{n+1}{2}\right\rceil $. Then
    there exists an interval $ [a,b]\subset (0,\infty) $ and a measure $ \mu\in\mathcal{M}_{[a,b]}(\mathbf{s}) $ satisfying 
    \begin{equation*}
        \infty > \ind_{\infty}(\mu) > \left\lceil \frac{n+1}{2}\right\rceil \ge \frac{n+1}{2}.
    \end{equation*} Enlarging the interval if necessary, we can assume that all atoms of $ \mu $ are in the interior of $ [a,b] $.
    Therefore, $ \ind_{[a,b]}(\mu) = \ind_{\infty}(\mu) > \frac{n+1}{2} $.
    By \cite[Theorem III.4.1]{KN1}, this implies that $ \mathbf{s} $ is strictly positive on $ [a,b] $, contradictorily to our assumption. The proof of (ii) is a straightforward modification of the above reasoning (we leave details to the reader).
\end{proof}
Below we present one of the main results in this section, which reveals the relationship between index and strict positivity.
\begin{theorem}
    \label{ThmStrictlyPositiveSequencesInfiniteInterval}
    Let $ n\in\mathbb{N} $ and let $ I $ be either $ (0,\infty) $ or $ (0,1] $. For a positive sequence $ \mathbf{s} = (s_0,s_1,\ldots,s_n)\subset
    [0,\infty) $ on $ I $ the following conditions are equivalent:
    \begin{enumerate}
        \item $ \mathbf{s} $ is strictly positive on $ I $,
        \item $ \ind_{I}(\mathbf{s}) = \begin{cases}
            \left\lceil \frac{n+1}{2}\right\rceil, & \text{if } I = (0,\infty),\\
            \frac{n+1}{2}, & \text{if } I = (0,1]
        \end{cases} $,
        \item $ \mathbf{s} $ is indeterminate on $ I $.
    \end{enumerate}
\end{theorem}
\begin{proof}
    We will show the equivalence in the case $ I = (0,\infty) $. The proof in the case $ I=(0,1] $ (as a straightforward modification of the presented reasoning) is left to the reader.\\ 
    (i)$ \Longrightarrow $(ii). From Lemma \ref{LemIndexEstimationInfiniteInterval} it follows that $
    \ind_{\infty}(\mathbf{s}) \le \left\lceil \frac{n+1}{2}\right\rceil $.
    Suppose to the contrary that $ \ind_{\infty}(\mathbf{s}) < \left\lceil
    \frac{n+1}{2}\right\rceil $. Let $ J_{1}\subset (0,\infty) $ be a compact interval such that $ \mathbf{s} $ is
    strictly positive on $ J_{1} $. Hence, by \cite[Theorems III.4.1 and III.5.1]{KN1}, $
    \ind_{J_{1}}(\mathbf{s}) = \frac{n+1}{2} $.  Take any $
    \mu\in\mathcal{M}_{\infty}(\mathbf{s}) $ satisfying $ \ind_{\infty}(\mu)
    = \ind_{\infty}(\mathbf{s}) $. Then $ \mu\in\mathcal{M}_{J_{2}}(\mathbf{s}) $ for some compact interval $ J_{2}\subset (0,\infty) $; without loss of generality we can assume that all atoms of $ \mu $ are in the interior of $ J_{2} $. Take a compact interval $ J\subset (0,\infty) $ such that $
    J_{1}\cup J_{2}\subset J $. Since $ \mathbf{s} $ is strictly positive on $ J $, by \cite[Theorem III.4.1]{KN1} we have that every $ \nu\in
    \mathcal{M}_{J}(\mathbf{s}) $ satisfies $ \ind_{J}(\nu) \ge
    \frac{n+1}{2} $. We also have $ \mu\in\mathcal{M}_{J}(\mathbf{s}) $ and $
    \supp\mu\subset \intr J $. If $ n=2m+1 $, $
    m\in\mathbb{N} $, then
    \begin{equation*}
        \ind_{J}(\mu) =
        \ind_{\infty}(\mu) < \frac{n+1}{2}.
    \end{equation*} If $ n = 2m $, $ m\in\mathbb{N} $,
    then $ \ind_{\infty}(\mu) \le \frac{n}{2} $, so
    \begin{equation*}
        \ind_{J}(\mu) = \ind_{\infty}(\mu) \le \frac{n}{2} < \frac{n+1}{2}.
    \end{equation*}
    In both cases we get a contradiction with \cite[Theorem III.4.1]{KN1}. \\
    (ii)$ \Longrightarrow $(iii). Take $ \mu\in\mathcal{M}_{\infty}(\mathbf{s}) $
    such that $ \ind_{\infty}(\mu) = \left\lceil \frac{n+1}{2}\right\rceil $. Then $ \mu\in\mathcal{M}_{J}(\mathbf{s}) $ for some compact interval $ J \subset (0,\infty) $; without loss of generality we can
    assume $ \supp\mu \subset \intr J $. From this, it follows that 
    \begin{equation*}
        \ind_{J}(\mu) = \ind_{\infty}(\mu) = \left\lceil
        \frac{n+1}{2}\right\rceil \ge \frac{n+1}{2},
    \end{equation*} so, by \cite[Theorem III.4.1]{KN1}, $ \mathbf{s} $ is indeterminate on $ J $ and, consequently, indeterminate on $ (0,\infty) $.\\
    (iii)$ \Longrightarrow $(i). Suppose $ \mu_1,\mu_2\in\mathcal{M}_{\infty}(\mathbf{s}) $ are distinct. Without loss of generality we can assume
    that $ \mu_{1},\mu_{2} $ are supported in the common compact interval $ J\subset (0,\infty) $. By \cite[Theorem III.4.1]{KN1}, $ \mathbf{s} $ is strictly positive on $ J $, hence strictly positive on $ (0,\infty) $.
\end{proof}
\begin{corollary}
    \label{CorUniqueMinimalMeasureStrictlyPositiveInfiniteInterval}
    If $ n\in\mathbb{N} $ is odd and $ \mathbf{s} = (s_0,\ldots, s_n)\subset [0,\infty) $ is strictly positive on $ (0,\infty) $, then there exists the unique measure $ \mu\in\mathcal{M}_{\infty}(\mathbf{s}) $ with $ \ind_{\infty}(\mu) =  \left\lceil \frac{n+1}{2}\right\rceil $; its atoms are roots of the polynomial
    \begin{equation}
        \label{FormPolynomialOfMinimalMeasureOfStrictlyPositiveInfiniteInterval}
        Q(t) = \det\begin{bmatrix}
        s_0 & s_1 & \ldots & s_{m-1} & 1\\
        s_1 & s_2 & \ldots & s_{m} & t\\
        \vdots & \vdots & \vdots & \vdots & \vdots\\
        s_{m} & s_{m+1} & \ldots & s_{2m-1} & t^{m}
        \end{bmatrix},
    \end{equation}
    where $ n = 2m-1 $, $ m\in\mathbb{N}_{1} $.
\end{corollary}
\begin{proof}
    Suppose $ \mu\in\mathcal{M}_{\infty}(\mathbf{s}) $ is a measure satisfying $ \ind_{\infty}(\mu) = \left\lceil \frac{n+1}{2}\right\rceil = m $. Then $ \mu\in\mathcal{M}_{[a,b]}(\mathbf{s}) $ for some interval $ [a,b]\subset (0,\infty) $. Without loss of generality we can assume that $ \supp \mu \subset (a,b) $, which implies that $ \ind_{[a,b]}(\mu) = m $. Hence, $ \mu $ has to be a principal measure of $ \mathbf{s} $ on $ [a,b] $. From the fact that all roots are in $ (a,b) $, by \cite[Section III.\S 4.1]{KN1}, it follows that $ \mu $ is the lower principal measure of $ \mathbf{s} $. Therefore, by \cite[Section III.\S 5.3]{KN1}, atoms of $ \mu $ are roots of (\ref{FormPolynomialOfMinimalMeasureOfStrictlyPositiveInfiniteInterval}).
\end{proof}
When $ n $ is even, the situation is not so simple as in the above corollary. Namely, if $ [a,b]\subset (0,\infty) $ is such that $ \mathbf{s} $ is strictly positive on $ [a,b] $, then any principal measure $ \mu\in\mathcal{M}_{[a,b]}(\mathbf{s}) $ has a support of cardinality $ \left\lceil \frac{n+1}{2}\right\rceil $ and its atoms depend on the interval $ [a,b] $ (in particular, they are different for different intervals $ [a,b] $; see \cite[Section III.\S 5.3]{KN1}). Hence, in this case we have infinitely many measures with minimal number of atoms, because $ \mathbf{s} $ is strictly positive also on every interval $ (0,\infty)\supset [a',b']\supset [a,b] $. Nevertheless, we can parametrize the set of all these measures, what will be seen in the next section.\par 
In the case of interval $ (0,1] $, the situation is much simpler, beacause the measure attaining the index of a strictly positive sequence is always unique.
\begin{corollary}
    \label{CorUniqueMinimalMeasureStrictlyPositiveHalfOpenInterval}
    Let $ n\in\mathbb{N} $. If $ \mathbf{s} = (s_0,\ldots,s_n)\subset [0,\infty) $ is strictly positive on $ (0,1] $, then there exists the unique measure $ \mu\in\mathcal{M}_{1}(\mathbf{s}) $ of index $ \frac{n+1}{2} $. Moreover,
    \begin{enumerate}
        \item if $ n = 2m-1 $, $ m\in\mathbb{N}_{1} $, then atoms of $ \mu $ are roots of the polynomial
        \begin{equation}
            \label{FormPolynomialOfMinimalMeasureOddHalfOpenInterval}
            Q(t) = \det\begin{bmatrix}
            s_0 & s_1 & \ldots & s_{m-1} & 1\\
            s_1 & s_2 & \ldots & s_{m} & t\\
            \vdots & \vdots & \vdots & \vdots & \vdots\\
            s_{m} & s_{m+1} & \ldots & s_{2m-1} & t^{m}
            \end{bmatrix},
        \end{equation}
        \item if $ n= 2m $, $ m\in\mathbb{N} $, then atoms of $ \mu $ are roots of the polynomial
        \begin{equation}
            \label{FormPolynomialOfMinimalMeasureEvenHalfOpenInterval}
            Q(t) = (1-t)\det\begin{bmatrix}
            s'_{0} & s'_{1} & \ldots & s'_{m-1} & 1\\
            s'_{1} & s'_{2} & \ldots & s'_{m} & t\\
            \vdots & \vdots & \vdots & \vdots & \vdots\\
            s'_{m} & s'_{m+1} & \ldots & s'_{2m-1} &
            t^{m}
            \end{bmatrix},
        \end{equation}
        where $ s'_{k} = s_{k}-s_{k+1} $, $ k\in\mathbb{N}\cap[0,2m-1] $.
    \end{enumerate}
\end{corollary}
\begin{proof}
    Suppose $ \mu\in\mathcal{M}_{1}(\mathbf{s}) $ is a measure satisfying $ \ind_{1}(\mu) =  \frac{n+1}{2} $. Then $ \mu\in\mathcal{M}_{[a,1]}(\mathbf{s}) $ for some interval $ [a,1]\subset (0,1] $. Without loss of generality we can assume that $ \supp\mu \subset (a,1] $, so that $ \ind_{[a,1]}(\mu) = \frac{n+1}{2} $. Hence, $ \mu $ has to be a principal measure of $ \mathbf{s} $ on $ [a,1] $. If $ n = 2m-1 $, $ m\in\mathbb{N}_{1} $, then, by \cite[Section III.\S 4.1]{KN1}, $ \mu $ is the lower principal measure and, in view of \cite[Section III.\S 5.3]{KN1}, its atoms are roots of (\ref{FormPolynomialOfMinimalMeasureOddHalfOpenInterval}). If $ n = 2m $, $ m\in\mathbb{N} $, then, again by \cite[Section III.\S 4.1]{KN1}, $ \mu $ is the upper principal measure of $ \mathbf{s} $ and, in view of \cite[Section III.\S 5.3]{KN1}, its atoms are roots of (\ref{FormPolynomialOfMinimalMeasureEvenHalfOpenInterval}).
\end{proof}

	\section{Backward extensions of moment sequences on $ (0,\infty) $ and $ (0,1] $}
\label{SecBackwardExtensionsInfiniteInterval}
The question of backward extendibility of (infinite) moment sequence has a well-known answer. In \cite{Wri1} backward extensions were studied using methods of continued fractions. In \cite{Sza1} the methods of reproducing kernel Hilbert spaces are involved in the solution of the backward extendibility problem for (infinite) moment sequence. In this section, we present a new approach to characterization of backward extensions of (truncated) moment sequences using the theory developed in Section \ref{SecMomentsOnInfiniteInterval}.\par
\begin{definition}
    \label{DefBackwardExtensionInfiniteInterval}
    Let $ I\subset \mathbb{R} $ be either $ (0,\infty) $ or $ (0,1] $. Assume $ n\in\mathbb{N} $ and let $ \mathbf{s} = (s_0,s_1,\ldots,s_n)\subset [0,\infty) $ be positive on $ I $. Suppose $ k\in\mathbb{N}_{1} $ and $ s_{-k},\ldots,s_{-1}\in [0,\infty) $. We say that a sequence $ \mathbf{s}' = (s_{-k},\ldots,s_{-1},s_0,\ldots, s_n) $ is a \textit{backward extension} of $ \mathbf{s} $ on $ I $ if $ \mathbf{s}' $ is positive on $ I $.
\end{definition}
Since, by Theorem \ref{ThmStrictlyPositiveSequencesInfiniteInterval}, singularly positive sequences are determinate on $ I $, backward extensions of such sequences are uniquely determined by $ \int_{I} \frac{1}{t}\ddd\mu(t) $, where
$ \mu $ is the unique representing measure. In what follows, we restrict our considerations on backward extensions to the case of strictly positive sequences.\par
If $ \mathbf{s}= (s_0,\ldots,s_n)\subset [0,\infty) $, $ n\in\mathbb{N} $, is a strictly positive moment sequence on $ (0,\infty) $, then we denote\footnote{We stick to the convention that $ \inf \varnothing = \infty $ and $ \sup\varnothing= -\infty $.}
\begin{align*}
    t_{a,b}(\mathbf{s}) &= \inf\left\{\int_{[a,b]}\frac{1}{t}\ddd\mu(t)\!: \mu\in
\mathcal{M}_{[a,b]}(\mathbf{s}) \right\}, & a,b\in(0,\infty),\ a<b,\\
    t_{\infty}(\mathbf{s}) &= \inf_{0<a<b} t_{a,b}(\mathbf{s}), \\
    T_{a,b}(\mathbf{s}) &= \sup\left\{\int_{[a,b]}\frac{1}{t}\ddd\mu(t)\!: \mu\in
    \mathcal{M}_{[a,b]}(\mathbf{s}) \right\}, & a,b\in(0,\infty),\ a<b,\\
    T_{\infty}(\mathbf{s}) &= \sup_{0<a<b} T_{a,b}(\mathbf{s}).
\end{align*}
If $ \mathbf{s} $ is strictly positive on $ (0,1] $, then we set
\begin{align*}
    t_{1}(\mathbf{s}) &= \inf_{0<a<b\le 1} t_{a,b}(\mathbf{s}),\\
    T_{1}(\mathbf{s}) &= \sup_{0<a<b\le 1} T_{a,b}(\mathbf{s}).
\end{align*}
Observe that $ t_{a,b}(\mathbf{s}) $ decreases, when $ b $ increases and $ T_{a,b}(\mathbf{s}) $ increases, when $ a $ decreases. The next technical lemma states that, actually, $ T_{a,b}(\mathbf{s}) $ increases to $ \infty $, which will be crucial in the subsequent results. Before the proof, we need a formulas for limiting values of integral of the reciprocal function on compact intervals. These formulas were given in \cite[Section IV.\S 2.3]{KN1}; unfortunately, they are wrong. The authors used the fact that if $ a,b,x\in\mathbb{R} $, $ x<a<b $, and $ n\in \mathbb{N} $, then the condition $ T_{+}(U) $ (see \cite[p. 109]{KN1}) holds for the function $ (-1)^{n+1}\Omega(t) = \frac{(-1)^{n+1}}{x-t} $, which is not true. Using \cite[P.1.1]{KN1} it can be shown that the condition $ T_{+}(U) $ (see \cite[p. 109]{KN1}) holds for the function $ (-1)^{n}\Omega(t) = \frac{(-1)^{n}}{x-t} $. Hence, if $ \mathbf{s} = (s_{0},s_{1},\ldots,s_{n}) \subset [0,\infty) $ is a strictly positive moment sequence on the interval $ [a,b] $, then the equality (2.10) in \cite[p.116]{KN1} should take the form
\begin{equation*}
    (-1)^{n}\frac{\underline{P}(x)}{\underline{Q}(x)} \le (-1)^{n}\int_{a}^{b} \frac{1}{x-t}\ddd\mu(t) \le (-1)^{n}\frac{\overline{P(x)}}{\overline{Q}(x)}, \qquad \mu\in \mathcal{M}_{[a,b]}(\mathbf{s}),
\end{equation*}
where $ \underline{Q},\overline{Q} $ are polynomials with roots precisely at atoms of lower and upper principal measure, respectively (see \cite[Section III.\S 5.3]{KN1}), and
\begin{align*}
    \underline{P}(x) &=
    \sigma\left(\frac{\underline{Q}(t)-\underline{Q}(x)}{t-x}\right),\\
    \overline{P}(x) &= \sigma\left(\frac{\overline{Q}(t)-\overline{Q}(x)}{t-x}\right).
\end{align*}
Taking into account the rest of the reasoning in \cite[Section IV.\S 2.3]{KN1}, we obtain that for $ x = 0 $,
\begin{align}
    \label{FormBackwardExtensionCompactIntervalInfimum}
    \inf\left\{ (-1)^{n+1}\int_{[a,b]}\frac{1}{t}\ddd\mu(t)\!: \mu\in \mathcal{M}_{[a,b]}(\mathbf{s}) \right\} = (-1)^{n}\frac{\underline{P}(0)}{\underline{Q}(0)}, \\
    \label{FormBackwardExtensionCompactIntervalSupremum}
    \sup\left\{ (-1)^{n+1}\int_{[a,b]}\frac{1}{t}\ddd\mu(t)\!: \mu\in \mathcal{M}_{[a,b]}(\mathbf{s}) \right\} = (-1)^{n}\frac{\overline{P}(0)}{\overline{Q}(0)},
\end{align}
What is more, since $ \mathcal{M}_{[a,b]}(\mathbf{s}) $ is convex and compact in the weak topology (the latter can be deduced from Banach-Alaoglu and Riesz representation theorems), we have
\begin{equation*}
    \left\{ (-1)^{n+1}\int_{[a,b]}\frac{1}{t}\ddd\mu(t)\!: \mu\in \mathcal{M}_{[a,b]}(\mathbf{s}) \right\} = \left[(-1)^{n}\frac{\underline{P}(0)}{\underline{Q}(0)}, (-1)^{n}\frac{\overline{P}(0)}{\overline{Q}(0)}\right].
\end{equation*}

\begin{lemma}
    \label{LemSupremumOfSupremumsInfiniteInterval}
    Let $ n\in\mathbb{N} $. If $ \mathbf{s} = (s_0,s_1,\ldots,s_n) \subset [0,\infty) $ is strictly positive on $ (0,\infty) $ $ ( $resp. on $ (0,1] $$ ) $, then $ T_{\infty}(\mathbf{s}) = \infty $ $ ( $resp. $ T_{1}(\mathbf{s}) = \infty $$ ) $.
\end{lemma}
\begin{proof}
    Assume that $ \mathbf{s} $ is strictly positive on $ (0,\infty) $. Fix $ b\in (0,\infty) $ such that $ \mathbf{s} $ is strictly positive on $ [a,b] $ for some $ a\in (0,b) $. Then, it holds $ T_{a,b}(\mathbf{s}) > -\infty $. By the previous remark, it is enough to show
    that $ T_{a,b}(\mathbf{s}) \stackrel{a\to0+}{\longrightarrow}\infty $.  For the further use, let us introduce the following notation: if $ N\in\mathbb{N} $ is odd and $ \mathbf{t} = (t_0,\ldots,t_N) \subset \mathbb{R} $ and $ \mathbf{f} =  (f_{0},\ldots, f_{\frac{N+1}{2}})\subset \mathbb{R} $, then we denote
    \begin{align*}
        \mathcal{D}(\mathbf{t}, \mathbf{f}) &= \det \begin{bmatrix}
        t_0 & t_1 & \ldots & t_{\frac{N+1}{2}-1} & f_0 \\
        t_1 & t_2 & \ldots & t_{\frac{N+1}{2}} & f_1\\
        \vdots & \vdots & \vdots & \vdots & \vdots \\
        t_{\frac{N+1}{2}} & t_{\frac{N+1}{2}+1} & \ldots & t_{N} & f_{\frac{N+1}{2}}
        \end{bmatrix}
    \end{align*}
    and
    \begin{align*}
        \mathcal{E}(\mathbf{t}, \mathbf{f}) &= \det\begin{bmatrix}
        f_0 & t_0 & t_1 & \ldots & t_{\frac{N+1}{2}-1}\\
        f_1 & t_1 & t_2 & \ldots & t_{\frac{N}{2}}\\
        \vdots & \vdots & \vdots & \vdots & \vdots \\
        f_{\frac{N+1}{2}}& t_{\frac{N+1}{2}} & t_{\frac{N+1}{2}+1} & \ldots & t_{N}
        \end{bmatrix}
    \end{align*}
    Case 1: $ n = 2m-1 $, $ m\in\mathbb{N}_{1} $. In this case, by \eqref{FormBackwardExtensionCompactIntervalSupremum}, $ T_{a,b} =
    -\frac{\overline{P}(0)}{\overline{Q}(0)} $, where $ a\in (0,b) $ is such that $ \mathbf{s} $ is strictly positive on $ [a,b] $. Applying the formula for $ \overline{Q} $ from \cite[Section III.\S 5.3]{KN1}, by multilinearity of the determinant, we obtain the following:
    \begin{align*}
        \overline{P}(0) &= \sigma\left(\frac{\overline{Q}(t)-\overline{Q}(0)}{t}\right) = \mathcal{D}\left((s_{k}')_{k=0}^{2m-3}, ((s_{k-1}'))_{k=0}^{m-1} \right),
    \end{align*}
    where $ s_{k}' = (a+b)s_{k+1}-abs_{k}-s_{k+2} $ for $ k\in\mathbb{N}\cap [0,2m-2] $ and $ s_{-1}=0 $. Set $ s^{(1)}_{k} = s_{k+1}-bs_{k} $ for $ k\in\mathbb{Z}\cap [-1,2m-2] $. Then, again by multilinearity of the determinant, we obtain
    \begin{align}
        \label{FormUpperBoundIntoTwoComponentsOddCase}
        -\frac{\overline{P}(0)}{\overline{Q}(0)} &= \frac{\mathcal{D}\left((s_{k}')_{k=0}^{2m-3}, (s_{k-1}')_{k=0}^{m-1} \right)}{ab\mathcal{D}\left((s_{k}')_{k=0}^{2m-3}, (1,0,\ldots,0) \right)}\\
        \notag &= \frac{\mathcal{D}\left((s_{k}')_{k=0}^{2m-3}, (as^{(1)}_{k-1})_{k=0}^{m-1} \right)}{ab\mathcal{D}\left((s_{k}')_{k=0}^{2m-3}, (1,0,\ldots,0) \right)} + \frac{\mathcal{D}\left((s_{k}')_{k=0}^{2m-3}, (-s^{(1)}_{k})_{k=0}^{m-1} \right)}{ab\mathcal{D}\left((s_{k}')_{k=0}^{2m-3}, (1,0,\ldots,0) \right)} 
    \end{align}
    Now, we calculate the limit of the first term of the right hand side of \eqref{FormUpperBoundIntoTwoComponentsOddCase}. Again, by basic properties of the determinant, we have
    \begin{align*}
        \frac{\mathcal{D}\left((s_{k}')_{k=0}^{2m-3}, (as^{(1)}_{k-1})_{k=0}^{m-1} \right)}{ab\mathcal{D}\left((s_{k}')_{k=0}^{2m-3}, (1,0,\ldots,0) \right)} &= \frac{\mathcal{D}\left((s_{k}')_{k=0}^{2m-3}, (s^{(1)}_{k-1})_{k=0}^{m-1} \right)}{b\mathcal{D}\left((s_{k}')_{k=0}^{2m-3}, (1,0,\ldots,0) \right)}.
    \end{align*}
    Observe that continuity of the determinant gives us
    \begin{equation*}
        \mathcal{D}\left((s_{k}')_{k=0}^{2m-3}, (1,0,\ldots,0) \right) \stackrel{a\to 0+}{\longrightarrow} \mathcal{D}\left((-s^{(1)}_{k+1})_{k=0}^{2m-3}, (1,0,\ldots,0) \right)
    \end{equation*}
    and that the above limit is non-zero by \cite[Theorem 2.4]{KN1} and \cite[Remark 2.1]{KN1}. Summarizing, we have
    \begin{equation}
        \label{FormUpperBoundFirstTermOddCase}
        \frac{\mathcal{D}\left((s_{k}')_{k=0}^{2m-3}, (as^{(1)}_{k-1})_{k=0}^{m-1} \right)}{ab\mathcal{D}\left((s_{k}')_{k=0}^{2m-3}, (1,0,\ldots,0) \right)}\stackrel{a\to 0+}{\longrightarrow} \frac{\mathcal{D}\left((s^{(1)}_{k+1})_{k=0}^{2m-3}, (s^{(1)}_{k-1})_{k=0}^{m-1} \right)}{b\mathcal{D}\left((-s^{(1)}_{k+1})_{k=0}^{2m-3}, (1,0,\ldots,0) \right)}
    \end{equation}
    Next, we calculate the limit of the second term of the right hand sided of \eqref{FormUpperBoundIntoTwoComponentsOddCase}. By the properties of the determinant, we see that
    \begin{align}
        \label{FormUpperBoundSecondTermNOddCase}
        % \frac{\mathcal{D}\left((s_{k}')_{k=0}^{2m-3}, (-s_{k+1}+bs_{k})_{k=0}^{m-1} \right)}{ab\mathcal{D}\left((s_{k}')_{k=0}^{2m-3}, (1,0,\ldots,0) \right)}
        \mathcal{D}\left((s_{k}')_{k=0}^{2m-3}, (-s^{(1)}_{k})_{k=0}^{m-1} \right) &\stackrel{a\to0+}{\longrightarrow} \mathcal{D}\left((-s^{(1)}_{k+1})_{k=0}^{2m-3}, (-s^{(1)}_{k})_{k=0}^{m-1} \right)\\
        \notag &= (-1)^{m}\mathcal{D}\left((-s^{(1)}_{k})_{k=0}^{2m-3}, (-s^{(1)}_{k})_{k=m}^{2m-1} \right)
    \end{align}
    where, by \cite[Theorem III.2.4]{KN1} and \cite[Remark III.2.1]{KN1},
    \begin{equation}
        \mathcal{D}\left((-s^{(1)}_{k})_{k=0}^{2m-3}, (-s^{(1)}_{k})_{k=m}^{2m-1} \right) > 0.
    \end{equation} Similarly, we have
    \begin{align}
        \label{FormUpperBoundSecondTermDOddCase}
        \mathcal{D}\left((s_{k}')_{k=0}^{2m-3}, (1,0,\ldots,0) \right) &\stackrel{a\to0+}{\longrightarrow} \mathcal{D}\left((-s^{(1)}_{k+1})_{k=0}^{2m-3}, (1,0,\ldots,0) \right)\\
        \notag &= (-1)^{m}\mathcal{E}\left((-s^{(1)}_{k+1})_{k=0}^{2m-3}, (1,0,\ldots,0) \right),
    \end{align}
    where, arguing as before, $ \mathcal{E}\left((-s^{(1)}_{k+1})_{k=0}^{2m-3}, (1,0,\ldots,0) \right) > 0 $. Using $ \eqref{FormUpperBoundSecondTermNOddCase} $ and \eqref{FormUpperBoundSecondTermDOddCase}, we obtain
    \begin{align*}
        \frac{\mathcal{D}\left((s_{k}')_{k=0}^{2m-3}, (-s^{(1)}_{k})_{k=0}^{m-1} \right)}{b\mathcal{D}\left((s_{k}')_{k=0}^{2m-3}, (1,0,\ldots,0) \right)} \stackrel{a\to0+}{\longrightarrow} \frac{\mathcal{D}\left((-s^{(1)}_{k})_{k=0}^{2m-3}, (-s^{(1)}_{k})_{k=m}^{2m-1} \right)}{b\mathcal{E}\left((-s^{(1)}_{k+1})_{k=0}^{2m-3}, (1,0,\ldots,0) \right)} >0.
    \end{align*}
    Since $ \frac{1}{a} \to \infty $ as $ a\to 0+ $, it follows that
    \begin{align}
        \label{FormUpperBoundSecondTermOddCase}
        \frac{\mathcal{D}\left((s_{k}')_{k=0}^{2m-3}, (-s^{(1)}_{k})_{k=0}^{m-1} \right)}{ab\mathcal{D}\left((s_{k}')_{k=0}^{2m-3}, (1,0,\ldots,0) \right)} \stackrel{a\to 0+}{\longrightarrow} \infty.
    \end{align}
    Combining \eqref{FormUpperBoundIntoTwoComponentsOddCase}, \eqref{FormUpperBoundFirstTermOddCase} and \eqref{FormUpperBoundSecondTermOddCase}, we get that $ T_{a,b}(\mathbf{s}) \stackrel{a\to0+}{\longrightarrow}
    \infty $.\\
    Case 2: $ n = 2m $, $ m\in\mathbb{N} $. In this case, by \eqref{FormBackwardExtensionCompactIntervalInfimum}, $ T_{a,b}(\mathbf{s}) = -\frac{\underline{P}(0)}{\underline{Q}(0)} $. Applying the formula for $ \underline{Q} $ from \cite[Section III.\S 5.3]{KN1}, by multilinearity of the determinant, we obtain
    \begin{align*}
        \underline{P}(0) = \sigma\left(\frac{\underline{Q}(t)-\underline{Q}(0)}{t} \right) = \mathcal{D}\left((s_{k+1}-as_{k})_{k=0}^{2m-1}, (s_k-as_{k-1})_{k=0}^{m} \right),
    \end{align*}
    where $ s_{-1} =0  $. Set $ s^{(1)}_{k} = s_{k+1}-as_{k} $ for $ k\in\mathbb{Z}\cap [-1,2m-2] $.
    Then, again using the formula from \cite[Section III.\S 5.3]{KN1}, by properties of the determinant, we see that
    \begin{align}
        \label{FormUpperBoundIntoTwoComponentsEvenCase}
        -\frac{\underline{P}(0)}{\underline{Q}(0)} &= \frac{\mathcal{D}\left((s^{(1)}_{k})_{k=0}^{2m-1}, (-as_{k-1})_{k=0}^{m} \right)}{a\mathcal{D}\left((s^{(1)}_{k})_{k=0}^{2m-1}, (1,0,\ldots,0)\right)} \\
        \notag &+ \frac{\mathcal{D}\left((s^{(1)}_{k})_{k=0}^{2m-1}, (s_k)_{k=0}^{m} \right)}{a\mathcal{D}\left((s^{(1)}_{k})_{k=0}^{2m-1}, (1,0,\ldots,0)\right)}
    \end{align}
    Now, we compute the limit of the first term of the right hand sided of \eqref{FormUpperBoundIntoTwoComponentsEvenCase}. Arguing similarly as to obtain \eqref{FormUpperBoundFirstTermOddCase}, we have
    \begin{align*}
    \frac{\mathcal{D}\left((s^{(1)}_{k})_{k=0}^{2m-1}, (-as_{k-1})_{k=0}^{m} \right)}{a\mathcal{D}\left((s^{(1)}_{k})_{k=0}^{2m-1}, (1,0,\ldots,0)\right)} &= \frac{\mathcal{D}\left((s^{(1)}_{k})_{k=0}^{2m-1}, (-s_{k-1})_{k=0}^{m} \right)}{\mathcal{D}\left((s^{(1)}_{k})_{k=0}^{2m-1}, (1,0,\ldots,0)\right)}\\
    &\stackrel{a\to0+}{\longrightarrow} \frac{\mathcal{D}\left((s_{k+1})_{k=0}^{2m-1}, (-s_{k-1})_{k=0}^{m} \right)}{\mathcal{D}\left((s_{k+1})_{k=0}^{2m-1}, (1,0,\ldots,0)\right)} > 0.
    \end{align*}
    Next, we compute the limit of the second term of the right hand sided of \eqref{FormUpperBoundIntoTwoComponentsEvenCase}. Using the same reasoning as to obtain \eqref{FormUpperBoundSecondTermNOddCase} and \eqref{FormUpperBoundSecondTermDOddCase}, we get
    \begin{align*}
        \mathcal{D}\left((s^{(1)}_{k})_{k=0}^{2m-1}, (s_k)_{k=0}^{m} \right) &\stackrel{a\to 0+}{\longrightarrow} \mathcal{D}\left((s_{k+1})_{k=0}^{2m-1}, (s_k)_{k=0}^{m} \right)\\
        \notag &= (-1)^{m}\mathcal{D}\left((s_{k})_{k=0}^{2m-1}, (s_k)_{k=m}^{2m} \right),
    \end{align*}
    and
    \begin{align*}
        \mathcal{D}\left((s^{(1)}_{k})_{k=0}^{2m-1}, (1,0,\ldots,0)\right) &\stackrel{a\to0+}{\longrightarrow} \mathcal{D}\left((s_{k+1})_{k=0}^{2m-1}, (1,0,\ldots,0)\right)\\
        \notag &= (-1)^{m}\mathcal{E}\left((s_{k+1})_{k=0}^{2m-1}, (1,0,\ldots,0)\right).
    \end{align*}
    By \cite[Theorem III.2.3]{KN1} and \cite[Remark III.2.1]{KN1}, we see that
    \begin{equation*}
        \mathcal{D}\left((s_{k})_{k=0}^{2m-1}, (s_k)_{k=m}^{2m} \right) >0, \ \mathcal{E}\left((s_{k+1})_{k=0}^{2m-1}, (1,0,\ldots,0)\right)>0.
    \end{equation*} Hence, arguing as to obtain \eqref{FormUpperBoundSecondTermOddCase}, we have
    \begin{align*}
        \frac{\mathcal{D}\left((s^{(1)}_{k})_{k=0}^{2m-1}, (s_k)_{k=0}^{m} \right)}{a\mathcal{D}\left((s^{(1)}_{k})_{k=0}^{2m-1}, (1,0,\ldots,0)\right)} \stackrel{a\to 0+}{\longrightarrow} \infty.
    \end{align*}
    Summarizing, we get $ T_{a,b}(\mathbf{s}) \stackrel{a \to 0+}{\longrightarrow}\infty $. If $ \mathbf{s} $ is assummed to be strictly positive on $ (0,1] $, then the same proof gives us the equality $ T_{1}(\mathbf{s}) = \infty $.
\end{proof}
The following lemma gives us the necessary condition on the one-step backward extension of a strictly positive sequence.
\begin{lemma}
    \label{LemWhereBackwardExtensionBelongsInfiniteInterval}
    Let $ n\in\mathbb{N} $. If $ \mathbf{s}' = (s_{-1},s_0,\ldots,s_n)\subset [0,\infty) $ is a backward extension of a strictly
    positive sequence $ \mathbf{s} = (s_0,s_1,\ldots,s_n) $ on $ (0,\infty) $ $ ( $resp. on $ (0,1] $$ ) $, then $ s_{-1}\in
    [t_{\infty}(\mathbf{s}), \infty) $ $ ( $resp. $ s_{-1}\in [t_{1}(\mathbf{s}), \infty) $$ ) $.
\end{lemma}
\begin{proof}
    Assume that $ \mathbf{s}' $ is strictly positive on $ (0,\infty) $ and let $ \mu\in \mathcal{M}_{\infty}(\mathbf{s}') $. Then $ \mu\in\mathcal{M}_{[a,b]}(\mathbf{s}') $ for some interval $ [a,b]\subset (0,\infty) $; without loss of generality we can assume that $ \supp\,\mu \subset (a,b) $. For $ \nu = t\ddd\mu(t) $ we have
    that $ \nu\in\mathcal{M}_{[a,b]}(\mathbf{s}) $, so
    \begin{equation*}
        s_{-1} = \int_{[a,b]}
        \frac{1}{t}\ddd\nu(t) \in [t_{a,b}(\mathbf{s}), T_{a,b}(\mathbf{s})] \subset
        [t_{\infty}(\mathbf{s}), \infty). \qedhere
    \end{equation*}
    With no substantial changes, the above proof works also when $ \mathbf{s} $ is assumed to be strictly positive on $ (0,1] $.
\end{proof}
We are ready to prove one of the main results of this section, which characterizes strictly positive backward extensions.
\begin{theorem}
    \label{ThmStrictlyPositiveBackwardExtensionsInfiniteInterval}
    Let $ n\in \mathbb{N} $ and let $ \mathbf{s} = (s_0,\ldots,s_n) \subset [0,\infty) $ be strictly
    positive on $ (0,\infty) $ $ ( $resp. on $ (0,1] $$ ) $. Assume $ s_{-1} \in [0,\infty) $. Then the following conditions are
    equivalent:
    \begin{enumerate}
        \item $ \mathbf{s}' = (s_{-1},s_0,\ldots,s_n) $ is a strictly positive backward extension of $ \mathbf{s} $ on $ (0,\infty) $ $ ( $resp. on $ (0,1] $$ ) $,
        \item $ s_{-1} \in (t_{\infty}(\mathbf{s}),\infty) $ $ ( $resp. $ s_{-1} \in (t_{1}(\mathbf{s}),\infty) $$ ) $.
    \end{enumerate}
\end{theorem}
\begin{proof}
    We concentrate on the case, when $ \mathbf{s} $ is strictly positive on $ (0,\infty) $.\\
    (i)$ \Longrightarrow $(ii). By Theorem \ref{ThmStrictlyPositiveSequencesInfiniteInterval}, $ 
    \ind_{\infty}(\mathbf{s}') = \left\lceil \frac{n+2}{2}\right\rceil $. Let $ \mu\in\mathcal{M}_{\infty}(\mathbf{s}') $ be any measure satisfying $ \ind_{\infty}(\mu) = \left\lceil
    \frac{n+2}{2}\right\rceil $. Then $ \mu\in\mathcal{M}_{[a,b]}(\mathbf{s}') $ for some interval $ [a,b]\subset(0,\infty) $; without loss of generality we can assume $ \supp\,\mu\subset (a,b) $. This implies that $ \ind_{[a,b]}(\mu) =
    \left\lceil \frac{n+2}{2}\right\rceil \ge \frac{n+2}{2} $. For $ \nu = t\ddd\mu(t)
    $ we see that $ \nu\in\mathcal{M}_{[a,b]}(\mathbf{s}) $ and $ \ind_{[a,b]}(\mu) = \ind_{[a,b]}(\nu) \ge
    \frac{n+2}{2} $. By \cite[Theorem IV.1.1]{KN1},
    \[ s_{-1} = \int_{[a,b]} \frac{1}{t}\ddd\nu(t) \in (t_{a,b}(\mathbf{s}),
    T_{a,b}(\mathbf{s})) \subset
    (t_{\infty}(\mathbf{s}),\infty). \]
    \\
    (ii)$ \Longrightarrow $(i). From Lemma \ref{LemSupremumOfSupremumsInfiniteInterval} it follows
    \begin{equation}
    \label{FormSumOfIntervalsInfiniteInterval}
    (t_{\infty}(\mathbf{s}),\infty) = \bigcup_{0<a<b}(t_{a,b}(\mathbf{s}),
    T_{a,b}(\mathbf{s})).
    \end{equation} By \eqref{FormSumOfIntervalsInfiniteInterval}, there exist $ a,b\in(0,\infty) $, $ a<b $, such
    that
    \begin{equation}
    \label{FormBackwardExtensionInOpenIntervalInfiniteInterval}
    s_{-1}\in (t_{a,b}(\mathbf{s}), T_{a,b}(\mathbf{s})).
    \end{equation}
    There exists a measure $ \mu\in\mathcal{M}_{[a,b]}(\mathbf{s}) $ satisfying $
    \int_{[a,b]}\frac{1}{t}\ddd\mu(t) = s_{-1} $. In view of \eqref{FormBackwardExtensionInOpenIntervalInfiniteInterval} and \cite[Theorem IV.1.1]{KN1}, the measure $ \mu $ is not a principal measure, so $
    \ind_{[a,b]}(\mu) \ge \frac{n+2}{2} $. For $ \nu = \frac{1}{t}\ddd\mu(t) $ we have $ \nu\in\mathcal{M}_{[a,b]}(\mathbf{s}') $ and $ \ind_{[a,b]}(\nu) \ge
    \frac{n+2}{2} = \frac{n+1+1}{2} $. By \cite[Theorem III.4.1]{KN1}, this implies that $ \mathbf{s}' $ is strictly
    positive on $ [a,b] $. Hence, $ \mathbf{s}' $ strictly positive on $ (0,\infty) $.
\end{proof}
It turns out that singularly positive backward extensions on $ (0,\infty) $ can appear only when $ n $ is odd.
\begin{corollary}
    \label{CorSingularlyPositiveBackwardExtensionsInfiniteInterval}
    Let $ n\in \mathbb{N} $ and let $ \mathbf{s} = (s_0,\ldots,s_n) \subset [0,\infty) $ be strictly
    positive on $ (0,\infty) $. Assume $ s_{-1} \in [0,\infty) $. Then the following conditions are
    equivalent:
    \begin{enumerate}
        \item $ \mathbf{s}' = (s_{-1},s_0,\ldots,s_n) $ is singularly positive
        backward extension of $ \mathbf{s} $ on $ (0,\infty) $,
        \item $ n $ is odd and $ s_{-1} = t_{\infty}(\mathbf{s}) $.
    \end{enumerate}
    Moreover, if (i) holds, then the atoms of the unique measure $ \mu\in\mathcal{M}_{\infty}(\mathbf{s}') $ are roots of the polynomial $ Q $, where
    \begin{equation}
    \label{FormPolynomialOfMinimalMeasureOfSingularlyPositiveBackwardExtensionInfiniteInterval}
    Q(t) = \det\begin{bmatrix}
    s_0 & s_1 & \ldots & s_{m-1} & 1\\
    s_1 & s_2 & \ldots & s_{m} & t\\
    \vdots & \vdots & \vdots & \vdots & \vdots\\
    s_{m} & s_{m+1} & \ldots & s_{2m-1} & t^{m}
    \end{bmatrix},
    \end{equation}
    and $ n = 2m-1 $, $ m\in\mathbb{N}_{1} $.
\end{corollary}
\begin{proof}
    By Lemma \ref{LemWhereBackwardExtensionBelongsInfiniteInterval} and Theorem \ref{ThmStrictlyPositiveBackwardExtensionsInfiniteInterval} the sequence $ \mathbf{s}' $ is singularly positive if
    and only if $ s_{-1} = t_{\infty}(\mathbf{s}) $. It is enough to prove that this
    can happen only if $ n $ is odd. From Theorem \ref{ThmStrictlyPositiveSequencesInfiniteInterval} we have $
    \ind_{\infty}(\mathbf{s}) = \left\lceil \frac{n+1}{2}\right \rceil $ and 
    \begin{equation*}
        \ind_{\infty}(\mathbf{s})\le \ind_{\infty}(\mathbf{s}') <
        \left\lceil \frac{n+2}{2}\right \rceil.
    \end{equation*} Thus, $
    \ind_{\infty}(\mathbf{s}') = \left\lceil \frac{n+1}{2}\right \rceil $.
    Moreover, $ \mathbf{s}' $ is determinate on $ (0,\infty) $. Let $
    \mu $ be the unique element of $ \mathcal{M}_{\infty}(\mathbf{s}') $; we can assume $ \supp\,\mu\subset (a,b) $ for some interval $ [a,b]\subset (0,\infty) $. Let $ \nu = t\ddd\mu(t) $. Then $
    \nu\in\mathcal{M}_{[a,b]}(\mathbf{s}) $ and 
    \begin{equation*}
        \ind_{[a,b]}(\nu) =
        \ind_{[a,b]}(\mu) = \ind_{\infty}(\mu) = \left\lceil
        \frac{n+1}{2}\right \rceil
    \end{equation*} Suppose to the contrary that $ n $ is even. Then $
    \ind_{[a,b]}(\nu) = \frac{n+2}{2} > \frac{n+1}{2} $, so, by \cite[Theorem IV.1.1]{KN1} we have
    \begin{equation*}
        s_{-1} = \int_{[a,b]}\frac{1}{t}\ddd\nu(t) > t_{a,b}(\mathbf{s}) \ge t_{\infty}(\mathbf{s}),
    \end{equation*} which gives us a contradiction. Hence, $ n $ is odd and, by Corollary \ref{CorUniqueMinimalMeasureStrictlyPositiveInfiniteInterval}, atoms of $ \nu $ (which are the same as atoms of $ \mu $) are roots of \eqref{FormPolynomialOfMinimalMeasureOfSingularlyPositiveBackwardExtensionInfiniteInterval}. Conversely, if $ n $ is odd, then $ t_{\infty}(\mathbf{s}) =
    t_{a,b}(\mathbf{s}) $ for every $ [a,b]\subset (0,\infty) $ such that $
    \mathbf{s} $ is strictly positive on $ [a,b] $, because atoms of the lower
    principal measure of $ \mathbf{s} $ do not depend on $ a $ and $ b $ (see \cite[Section III.\S 5.3]{KN1}). 
\end{proof}
In the case of interval $ (0,1] $ singularly positive backward extensions can appear in any case; we omit the proof, which is similar to the proof of Corollary \ref{CorSingularlyPositiveBackwardExtensionsInfiniteInterval}.
\begin{corollary}
    \label{CorSingularlyPositiveBackwardExtensionsHalfOpenInterval}
    Let $ n\in\mathbb{N} $. Let $ \mathbf{s} = (s_0,\ldots,s_n)\subset [0,\infty) $ be strictly positive on $ (0,1] $ and $ s_{-1}\in [0,\infty) $. Then the following are equivalent:
    \begin{enumerate}
        \item $ \mathbf{s}' = (s_{-1},\ldots,s_n) $ is a singularly positive backward extension of $ \mathbf{s} $,
        \item $ s_{-1} = t_{1}(\mathbf{s}) $,
    \end{enumerate}
    Moreover, if (i) holds, then the atoms of unique measure $ \mu\in\mathcal{M}_{1}(\mathbf{s}') $ are roots of the polynomial $ Q $, where 
    \begin{enumerate}
        \item       \begin{equation}
        \label{FormPolynomialOfSingularlyPositiveBackwardExtensionOddHalfOpenInterval}
        Q(t) = \det\begin{bmatrix}
        s_0 & s_1 & \ldots & s_{m-1} & 1\\
        s_1 & s_2 & \ldots & s_{m} & t\\
        \vdots & \vdots & \vdots & \vdots & \vdots\\
        s_{m} & s_{m+1} & \ldots & s_{2m-1} & t^{m}
        \end{bmatrix},
        \end{equation}
        if $ n = 2m-1 $, $ m\in\mathbb{N}_{1} $,
        \item \begin{equation}
        \label{FormPolynomialOfSingularlyPositiveBackwardExtensionEvenHalfOpenInterval}
        Q(t) = (1-t)\det\begin{bmatrix}
        s'_{0} & s'_{1} & \ldots & s'_{m-1} & 1\\
        s'_{1} & s'_{2} & \ldots & s'_{m} & t\\
        \vdots & \vdots & \vdots & \vdots & \vdots\\
        s'_{m} & s'_{m+1} & \ldots & s'_{2m-1} &
        t^{m}
        \end{bmatrix},
        \end{equation}
        if $ n=2m $, $ m\in\mathbb{N} $, where $ s'_{k} = s_{k}-s_{k+1} $, $ k\in\mathbb{N}\cap[0,2m-1] $.
    \end{enumerate}
\end{corollary}
Next result gives the characterization of moments of a positive sequence on $ (0,\infty) $ in terms of numbers $ t_{\infty}(\,\cdot\,) $. This will be crucial in solving completion problems.
\begin{theorem}
    \label{ThmConditionsOnMomentsOfMinimalMeasureInfiniteInterval}
    Let $ n\in \mathbb{N} $ and let $ \mathbf{s} = (s_0,\ldots,s_n) \subset [0,\infty) $ be positive on $ (0,\infty) $. Set $ K=\ind_{\infty}(\mathbf{s}) $ and $ N = 2K-1 $. Then
    \begin{enumerate}
        \item $ s_{k}\in
        (t_{\infty}((s_{k+1},\ldots,s_n)),\infty) $ for $ k\in\mathbb{N}\cap[n-N,n-1] $,
        \item $ s_{k} = t_{\infty}((s_{k+1},
        \ldots,s_{k+N+1})) $ for $ k\in\mathbb{N}\cap [0,n-N-1] $.
    \end{enumerate} 
\end{theorem}
\begin{proof}
    By Lemma \ref{LemIndexEstimationInfiniteInterval}, $ K\le \left\lceil \frac{n+1}{2}\right \rceil $. Let $ \mu\in\mathcal{M}_{\infty}(\mathbf{s}) $ be such that $ \ind_{\infty}(\mu) = K  $.\\
    (i). If $ k = n-N $, then $ n-k-1 $ is even and, by Lemma \ref{LemWhereBackwardExtensionBelongsInfiniteInterval} and Corollary \ref{CorSingularlyPositiveBackwardExtensionsInfiniteInterval}, $ s_{n-N} \in (t_{\infty}((s_{n-N+1},\ldots, s_n)),\infty) $. Assume $ k>n-N $. By Lemma \ref{LemWhereBackwardExtensionBelongsInfiniteInterval}, we have $ s_k\in
    [t_{\infty}((s_{k+1},\ldots,s_n)),\infty) $. Therefore, it is enough to prove that $
    s_k > t_{\infty}((s_{k+1},\ldots,s_n)) $. Suppose to the contrary that it is not
    the case. Then from Corollary \ref{CorSingularlyPositiveBackwardExtensionsInfiniteInterval} it follows that $ n-k-1 $ is
    odd and $ (s_k,\ldots,s_n) $ is singularly positive on $ (0,\infty) $ and, consequently, by Theorem \ref{ThmStrictlyPositiveSequencesInfiniteInterval},
    determinate on $ (0,\infty) $. Let $ \nu = t^{k}\ddd\mu(t) $. Then $ \ind_{\infty}(\mu)= \ind_{\infty}(\nu) < \frac{n-k}{2}$, because $ \nu $
    has to be the unique element of $ \mathcal{M}_{\infty}((s_k,\ldots,s_n)) $. However, we
    have $ \frac{n-k}{2} < \frac{N}{2} = K-\frac{1}{2} = K-1+\frac{1}{2} $, so $
    \ind_{\infty}(\mu) \le K-1 $, which is a contradiction.\\
    (ii). Let $ \nu = t^{k+1}\ddd\mu(t) $. Then $ \nu\in\mathcal{M}_{\infty}((s_{k+1},\ldots, s_{k+N+1})) $ and $
    \ind_{\infty}(\nu) = \ind_{\infty}(\mu) = K $. Applying Theorem
    \ref{ThmStrictlyPositiveSequencesInfiniteInterval} to $ \nu $, we have that $
    (s_{k+1},\ldots, s_{k+N+1}) $ is strictly positive on $ (0,\infty) $. Hence, 
    \begin{equation*}
        \ind_{\infty}((s_{k+1},\ldots, s_{k+N+1})) = \left\lceil
        \frac{N+1}{2}\right\rceil = K.
    \end{equation*} Then 
    \begin{equation*}
        \ind_{\infty}((s_{k},\ldots,s_{k+N+1})) \le K = \frac{N+1}{2} <
        \frac{N+2}{2} < \left\lceil \frac{N+2}{2}\right\rceil,
    \end{equation*} because $ t^{k}\ddd\mu(t) \in\mathcal{M}_{\infty}((s_{k},\ldots,s_{k+N+1})) $ has support of cardinality $ K $. By
    Theorem \ref{ThmStrictlyPositiveSequencesInfiniteInterval} the sequence $ (s_{k},\ldots,s_{k+N+1}) $ is
    singularly positive on $ (0,\infty) $. Hence, by Corollary \ref{CorSingularlyPositiveBackwardExtensionsInfiniteInterval}, $ s_{k} =
    t_{\infty}((s_{k+1},\ldots,s_{k+N+1})) $.
\end{proof}
The counterpart of the above theorem for moment sequence on $ (0,1] $ is presented below (we omit the proof, which is a straightforward modification of the proof of Theorem \ref{ThmConditionsOnMomentsOfMinimalMeasureInfiniteInterval}).
\begin{theorem}
    \label{ThmConditionsOnMomentsOfMinimalMeasureHalfOpenInterval}
    Let $ n\in\mathbb{N} $ and let $ \mathbf{s} = (s_0,\ldots,s_n)\subset [0,\infty) $ be positive on $ (0,1] $. Set $ K= \ind_{1}(\mathbf{s}) \le \frac{n+1}{2} $ and $ N=2K-1 $. Then
    \begin{enumerate}
        \item $ s_{k}\in t_{1}((s_{k+1},\ldots,s_{n})) $ for $ k\in\mathbb{N}\cap [n-N,n-1] $,
        \item $ s_{k} = t_{1}((s_{k+1},\ldots,s_{k+N+1})) $ for $ k\in\mathbb{N}\cap[0,n-N-1] $.
    \end{enumerate}
\end{theorem}
The following theorem can be regarded as the converse of the previous theorem: if we have numbers satisfying conditions as in Theorem \ref{ThmConditionsOnMomentsOfMinimalMeasureInfiniteInterval}, then they form a positive sequence on $ (0,\infty) $ with prescribed index. We state this result in a little bit more general setting.
\begin{theorem}
    \label{ThmExistenceOfBackwardExtensionWithPrescribedNumberOfAtomsInfiniteInterval}
    Let $ n\in \mathbb{N} $ and let $ \mathbf{s} = (s_0,\ldots,s_n) \subset [0,\infty) $ be strictly
    positive on $ (0,\infty) $. Suppose $ r\in \mathbb{N}_{1} $ and $ K\in\mathbb{N}\cap \left[\left\lceil
    \frac{n+1}{2}\right\rceil, \left\lceil \frac{n+r+1}{2}\right\rceil\right] $.
    Set $ N = 2K-1 $ and let $ s_{-r},\ldots,s_{-1} \in [0,\infty) $ be such that
    \begin{enumerate}
        \item $ s_{k}\in
        (t_{\infty}((s_{k+1},\ldots,s_0,s_1,\ldots,s_n)),\infty) $ for $ k\in \mathbb{Z}\cap[n-N,-1] $,
        \item $ s_{k}= t_{\infty}((s_{k+1},\ldots,s_{k+N+1})) $ for $ k \in\mathbb{Z}\cap [-r,n-N-1] $.
    \end{enumerate}
    Then $ \mathbf{s}' = (s_{-r},s_{-r+1},\ldots,s_n) $ is a backward extension of $
    \mathbf{s} $ on $ (0,\infty) $ satisfying $ \ind_{\infty}(\mathbf{s}') = K $. Moreover, if $ N\le n+r $, then the atoms of the unique measure $ \mu\in\mathcal{M}_{\infty}(\mathbf{s}\prime) $ satisfying $ \ind_{\infty}(\mu) = K $ are roots of the polynomial $ Q $, where
    \begin{equation}
    \label{FormPolynomialOfMinimalMeasureOfSequenceWithPrescribedNumberOfAtomsInfiniteInterval}
    Q(t) = \det\begin{bmatrix}
    s_{-r} & s_{-r+1} & \ldots & s_{-r+K-1} & 1\\
    s_{-r+1} & s_{-r+2} & \ldots & s_{-r+K} & t\\
    \vdots & \vdots & \vdots & \vdots & \vdots\\
    s_{-r+K} & s_{-r+K+1} & \ldots & s_{-r+2K-1} & t^{K}
    \end{bmatrix}.
    \end{equation}
\end{theorem}
\begin{proof}
    We show that for any $
    k\in\mathbb{Z}\cap [n-N,0] $ the sequence $ (s_{k}, s_{k+1},\ldots, s_n) $ is strictly positive on $ (0,\infty) $. The proof goes by induction. The case $ k = 0 $ follows
    from the assumption. If $ n-N\le k-1 < 0 $, then by inductive hypothesis $
    \mathbf{s}_k = (s_k,\ldots,s_n) $ is strictly positive on $ (0,\infty) $. Applying
    Theorem \ref{ThmStrictlyPositiveBackwardExtensionsInfiniteInterval} to $ \mathbf{s}_k $ and $ s_{k-1} $, we get
    that $ (s_{k-1},s_{k},\ldots,s_{n}) $ is strictly positive on $ (0,\infty) $. Next, from
    Corollary \ref{CorSingularlyPositiveBackwardExtensionsInfiniteInterval} applied to $ \mathbf{s}_1 = (s_{n-N},\ldots, s_{n}) $ and $
    s_{n-N-1} $, we derive that $ \mathbf{s}_1'=(s_{n-N-1},\ldots,s_n) $ is singularly
    positive on $ (0,\infty) $. Let $
    \mu $ be the unique element of $ \mathcal{M}_{\infty}(\mathbf{s}_1') $. Setting
    \begin{equation*}
        s_{k}' =
        \int_{(0,\infty)} t^{k+N-n+1}\ddd\mu(t), \qquad k\in\mathbb{Z}\cap [-r,n-N-2],
    \end{equation*} we obtain
    from Theorem \ref{ThmStrictlyPositiveSequencesInfiniteInterval} that $
    (s_{-r}',\ldots,s_{n-N-2}',s_{n-N-1},\ldots,s_n) $ is singularly positive on $ (0,\infty) $. The sequence $
    \mathbf{s}_1 $ is strictly positive on $ (0,\infty) $, so $ \ind_{\infty}(\mu) = \left\lceil
    \frac{N+1}{2}\right\rceil = K $. Therefore, using Corollary \ref{CorSingularlyPositiveBackwardExtensionsInfiniteInterval} we get
    \begin{equation*}
        s_{k}' = t_{\infty}((s_{k+1},\ldots,s_{k+N+1}) = s_{k}, \qquad k \in\mathbb{Z}\cap [-r,n-N-1].
    \end{equation*}
    This shows that $ \mathbf{s}' $ is a backward
    extension of $ \mathbf{s} $ on $ (0,\infty) $ of index $ K $. To prove the ''moreover'' part, consider $ \mathbf{s}'' = (s_{-r},\ldots,s_{-r+2K-1}) $. Then $ \ind_{\infty}(\mathbf{s}'') = \ind_{\infty}(\mathbf{s}') = K $. Hence, by Theorem \ref{ThmStrictlyPositiveSequencesInfiniteInterval}, $ \mathbf{s}'' $ is strictly positive and, by Corollary \ref{CorUniqueMinimalMeasureStrictlyPositiveInfiniteInterval}, atoms of the unique measure $ \mu\in\mathcal{M}_{\infty}(\mathbf{s}) $ satisfying $ \ind_{\infty}(\mu) = K $ are roots of \eqref{FormPolynomialOfMinimalMeasureOfSequenceWithPrescribedNumberOfAtomsInfiniteInterval}.
\end{proof}
Again, below we state (without the proof) the counterpart of the above theorem for moments on $ (0,1] $.
\begin{theorem}
    \label{ThmExistenceOfBackwardExtensionWithPrescribedNumberOfAtomsHalfOpenInterval}
    Let $ n\in\mathbb{N} $ and let $ \mathbf{s} = (s_0,s_1,\ldots,s_n)\subset [0,\infty) $ be strictly positive on $ (0,1] $. Suppose $ r\in \mathbb{N}_{1} $ and $ K\in \mathbb{Q} $ is such that $ \frac{n+1}{2} \le K\le \frac{n+r+1}{2} $ and $ 2K\in\mathbb{N} $. Set $ N = 2K-1 $ an let $ s_{-r},\ldots,s_{-1}\in [0,\infty) $ be such that
    \begin{enumerate}
        \item $ s_{k}\in (t_{1}((s_{k+1},\ldots,s_n)),\infty) $ for $ k\in\mathbb{Z}\cap[n-N,n-1] $,
        \item $ s_{k} = t_{1}((s_{k+1},\ldots,s_{k+N+1})) $ for $ k\in\mathbb{Z}\cap[-r,n-N-1] $.
    \end{enumerate}
    Then $ \mathbf{s}' = (s_{-r},\ldots,s_n) $ is a backward extension of $ \mathbf{s} $ on $ (0,1] $ and $ \ind_{1}(\mathbf{s}') = K $. Moreover, atoms of the unique measure $ \mu\in\mathcal{M}_{1}(\mathbf{s}') $ satisfying $ \ind_{1}(\mathbf{s}') = K $ are roots of the polynomial
    \begin{enumerate}
        \item       \begin{equation*}
            Q(t) = \det\begin{bmatrix}
            s_{-r} & s_{-r+1} & \ldots & s_{-r+K-1} & 1\\
            s_{-r+1} & s_{-r+2} & \ldots & s_{-r+K} & t\\
            \vdots & \vdots & \vdots & \vdots & \vdots\\
            s_{-r+K} & s_{-r+K+1} & \ldots & s_{-r+2K-1} & t^{m}
            \end{bmatrix},
            \end{equation*}
            if $ N = 2m-1 $, $ m\in\mathbb{N}_{1} $;
            \item \begin{equation*}
            Q(t) = (1-t)\det\begin{bmatrix}
            s'_{-r} & s'_{-r+1} & \ldots & s'_{-r+K-\frac{3}{2}} & 1\\
            s'_{-r+1} & s'_{-r+2} & \ldots & s'_{-r+K-\frac{1}{2}} & t\\
            \vdots & \vdots & \vdots & \vdots & \vdots\\
            s'_{-r+K-\frac{1}{2}} & s'_{-r+K+\frac{1}{2}} & \ldots & s'_{-r+2K-2} &
            t^{m}
            \end{bmatrix},
            \end{equation*}
            if $ N=2m $, $ m\in\mathbb{N} $.
    \end{enumerate}
\end{theorem}
At the end of the previous section we claimed that if $ n\in \mathbb{N} $ is even and $ \mathbf{s} = (s_0,\ldots,s_n)\subset [0,\infty] $ is strictly positive on $ (0,\infty) $, then we can parametrize the set of all measures in $ \mathcal{M}_{\infty}(\mathbf{s}) $ with support of minimal cardinality. We will do it now with help of our results on backward extensions.
\begin{corollary}
    \label{CorParametrizationOfMinimalMeasuresOfStrictlyPositiveSequenceEvenInfiniteInterval}
    Let $ \mathbf{s} = (s_0,\ldots,s_n)\in[0,\infty) $ be strictly positive on $ (0,\infty) $, where $ n\in\mathbb{N} $ is even. Denote by $ \mathcal{M}_{\infty}^{\min}(\mathbf{s}) $ the set of all measures $ \mu\in\mathcal{M}_{\infty}(\mathbf{s}) $ satisfying $ \ind_{\infty}(\mu) = \left\lceil \frac{n+1}{2}\right\rceil $. Then the mapping
    \begin{equation*}
        \Phi\!: \mathcal{M}_{\infty}^{\min}(\mathbf{s}) \ni \mu \longmapsto \int_{(0,\infty)} \frac{1}{t}\ddd\mu(t)\in (t_{\infty}(\mathbf{s}),\infty)
    \end{equation*}
    is a bijection. Moreover, if $ s_{-1} \in (t_{\infty}(\mathbf{s}),\infty) $, then atoms of $ \Phi^{-1}(s_{-1}) $ are roots of the polynomial
    \begin{equation}
        \label{FormPolynomialOfMinimalMeasureWithPrescribedBackwardExtensionEvenInfiniteInterval}
        Q(t) = \det\begin{bmatrix}
        s_{-1} & s_{0} & \ldots & s_{K-2} & 1\\
        s_{0} & s_{1} & \ldots & s_{K-1} & t\\
        \vdots & \vdots & \vdots & \vdots & \vdots\\
        s_{K-1} & s_{K} & \ldots & s_{2K-2} & t^{K}
        \end{bmatrix},
    \end{equation}
    where $ K = \left\lceil \frac{n+1}{2}\right\rceil $.
\end{corollary}
\begin{proof}
    First, observe that $ \Phi $ is well-defined. Indeed, if $ \mu\in\mathcal{M}_{\infty}(\mathbf{s}) $, then by Lemma \ref{LemWhereBackwardExtensionBelongsInfiniteInterval}, $ \int_{(0,\infty)} \frac{1}{t}\ddd\mu(t)\in [t_{\infty}(\mathbf{s}), \infty) $. Since $ n $ is even, in view of Corollary \ref{CorSingularlyPositiveBackwardExtensionsInfiniteInterval}, we have $ \int_{(0,\infty)} \frac{1}{t}\ddd\mu(t) > t_{\infty}(\mathbf{s}) $. Let $ \mu\in\mathcal{M}_{\infty}^{\min}(\mathbf{s}) $ and $ s_{-1}\in(t_{\infty}(\mathbf{s}),\infty) $. Applying Theorem \ref{ThmStrictlyPositiveSequencesInfiniteInterval} and Theorem \ref{ThmExistenceOfBackwardExtensionWithPrescribedNumberOfAtomsInfiniteInterval} with $ r = 1 $ and $ K = \left\lceil \frac{n+2}{2}\right\rceil = \left\lceil \frac{n+1}{2}\right\rceil $, we see that $ \mathbf{s}' = (s_{-1},s_{0},\ldots,s_{n}) $ is strictly positive on $ (0,\infty) $. Moreover, by Corollary \ref{CorUniqueMinimalMeasureStrictlyPositiveInfiniteInterval}, there exists the unique measure $ \nu_{s_{-1}}\in\mathcal{M}_{\infty}(\mathbf{s}) $ satisfying $ \ind_{\infty}(\nu_{s_{-1}}) = \left\lceil \frac{n+1}{2}\right\rceil $; its atoms are roots of \eqref{FormPolynomialOfMinimalMeasureOfSequenceWithPrescribedNumberOfAtomsInfiniteInterval}, which in this case becomes \eqref{FormPolynomialOfMinimalMeasureWithPrescribedBackwardExtensionEvenInfiniteInterval}. Then $ t\ddd\nu_{s_{-1}}(t)\in\mathcal{M}_{\infty}(\mathbf{s}) $ and $ \supp t\ddd\nu_{s_{-1}}(t) = \supp\,\nu_{s_{-1}} $. Define the mapping
    \begin{equation*}
        \Psi\!: (t_{\infty}(\mathbf{s}),\infty)\ni s_{-1}\to t\ddd\nu_{s_{-1}}(t)\in \mathcal{M}_{\infty}^{\min}(\mathbf{s}).
    \end{equation*} We show that $ \Psi\circ \Phi = \text{id}_{\mathcal{M}_{\infty}^{\min}(\mathbf{s})} $ and $ \Phi\circ\Psi = \text{id}_{(t_{\infty}(\mathbf{s}),\infty)} $. Let $ \mu\in\mathcal{M}_{\infty}^{\min}(\mathbf{s}) $. Set $ s_{-1} = \Phi(\mu) $. Because of the uniqueness of the measure $ \nu_{s_{-1}} $, we get that $ \frac{1}{t}\ddd\mu(t) = \nu_{s_{-1}} $. Hence, $ \mu = t\nu_{s_{-1}}(t) = \Psi(s_{-1}) $. Conversely, let $ s_{-1}\in (t_{\infty}(\mathbf{s}),\infty) $ and set $ \mu = \Psi(s_{-1}) = t\ddd\nu_{s_{-1}} $. Then $ \Phi(\mu) = \int_{(0,\infty)}\frac{1}{t}\ddd\mu(t) = \int_{(0,\infty)} \frac{1}{t} \cdot t\ddd\nu_{s_{-1}}(t) = s_{-1} $.
\end{proof}

	\section{Subnormal completions of weighted shifts on directed trees.}
\label{SecSubnormalCompletions}
In \cite[Chapter 6]{JJS1} the authors studied subnormal weighted shifts on directed trees and characterized subnormality by certain moment sequences (the obtained characterization involves also negative moments, see \cite[Theorem 6.2.1]{JJS1}). Using this apporach, in \cite{EJSY1} the authors investigated the subnormal completion problem for weighted shifts on trees with one branching point in general; in \cite{EJSY2} the same authors studied in details 2-generation subnormal completions (see \cite[Theorem 4.1]{EJSY2} for the solution of the 2-generation subnormal completion). In this section we will exploit the theory developed in the previous sections to obtain a solution of the subnormal completion problem for weighted shifts on directed trees with one branching point in full generality. \par
If $ \mathcal{T} = (V,E) $ is a directed tree, then \cite[Theorem 6.1.3]{JJS1} states that a bounded weighted shift $ S_{\pmb{\lambda}}\in\mathbf{B}(\ell^{2}(V)) $ on $ \mathcal{T} $ is subnormal if and only if the sequence $ (\lVert S^{n}_{\pmb{\lambda}} e_{u}\rVert^{2})_{n=0}^{\infty} $ is a Stieltjes moment sequence for every $ u\in V $. Moreover, the measure representing the moment sequence $ (\lVert S^{n}_{\pmb{\lambda}} e_{u}\rVert^{2})_{n=0}^{\infty} $ is unique; we denote it by $ \mu^{\pmb{\lambda}}_{u} $.\par
The next result generalizes \cite[Theorem 4.1]{EJSY2}. \par
\begin{theorem}
    \label{ThmSubnormalCompletion}
    Let $ \kappa\in\mathbb{N} $, $ \eta \in\overline{\mathbb{N}}_{2} $, $ p\in\mathbb{N}_{1} $. Let $ \pmb{\lambda}= \{\lambda_{-k}\}_{k=0}^{\kappa-1}\cup \{\lambda_{i,j} \}_{i,j=1}^{\eta,p} \subset (0,\infty) $ and let $ \{K_{i}\}_{i=1}^{\eta} \subset\mathbb{N}\cap [1,\left\lceil \frac{p+\kappa+1}{2}\right\rceil] $. Then the following conditions are equivalent:
    \begin{enumerate}
        \item there exists a subnormal completion $ S_{\pmb{\lambda'}}\in \mathbf{B}(\ell^{2}(V_{\eta,\kappa})) $ of $ \pmb{\lambda} $ such that for every $ i\in\mathbb{N}\cap[1,\eta] $ the measure $ \mu_{i,1}^{\pmb{\lambda'}} $ is $ K_{i} $-atomic,
        \item for every $ i\in\mathbb{N}\cap[1,\eta] $ there exists a sequence $ \{s_{i,-k}\}_{k=1}^{\kappa+1}\subset (0,\infty) $ such that for every $ i\in\mathbb{N}\cap[1,\eta] $
        \begin{align}
            \label{FormSubnormalCompletionStrictlyPositiveMoments}
            &s_{i,k} \in (t_{\infty}((s_{i,k+1},\ldots,s_{p-1})),+\infty), \qquad k\in\mathbb{Z}\cap [p-N_i-1,p-2],\\
            \label{FormSubnormalCompletionSingularlyPositiveMoments}
            &s_{i,k} = t_{\infty}((s_{i,k+1},\ldots,s_{k+N_i+1})), \qquad k\in\mathbb{Z}\cap [-\kappa-1,p-N_i-2]
        \end{align}
        and\footnote{Here and in the subsequent parts we follow the convention that $ \prod\varnothing = 1 $.}
        \begin{align}
            \label{FormSubnormalCompletionEquality}
            &\sum_{i=1}^{\eta} \lambda_{i,1}^{2}s_{i,-k-1} = \frac{1}{\prod_{j=0}^{k-1}\lambda_{-j}^{2}}, \qquad k\in\mathbb{N}\cap[0,\kappa-1],\\
            \label{FormSubnormalCompletionInequality}
            &\sum_{i=1}^{\eta} \lambda_{i,1}^{2}s_{i,-\kappa-1} \le \frac{1}{\prod_{j=0}^{\kappa-1}\lambda_{-j}^{2}},\\
            \label{FormSubnormalCompletionSupremumOfAtoms}
            &\sup_{i\in\mathbb{N}\cap [1,\eta]} -\frac{Q_{i}^{(K_i-1)}(0)}{Q_{i}^{(K_i)}(0)} < \infty,
        \end{align} 
        where $ s_{i,k} = \prod_{j=2}^{k+1} \lambda_{i,j}^{2} $, $ k\in \mathbb{N}\cap [0,p-1] $, $ N_i = 2K_i - 1 $ and $ Q_i\in\mathbb{R}_{K_{i}}[x] $ is a polynomial with simple zeros at atoms of $ {\mu\in\mathcal{M}_{\infty}((s_{i,-\kappa-1},\ldots,s_{i,p-1}))} $ satisfying $ \#\text{supp}\,\mu = K_i $ for $ i\in\mathbb{N}\cap[1,\eta] $.
    \end{enumerate}
\end{theorem}
\begin{proof}
    (i)$ \Longrightarrow $(ii). By (i) and \cite[Lemma 4.7]{EJSY1} we see that
    \begin{align}
        \label{FormConditionOnMoments1}
        &\int_{0}^{\infty} t^{k}\ddd\mu_i(t) = \prod_{j=2}^{k+1}\lambda_{i,j}^{2}, \qquad k\in\mathbb{N}\cap [0,p-1], \ i\in\mathbb{N}\cap [1,\eta],\\
        \label{FormConditionOnMoments2}
        &\sum_{i=1}^{\eta} \lambda_{i,1}^{2}\int_{0}^{\infty} \frac{1}{t^{k+1}}\ddd\mu_i(t) = \frac{1}{\prod_{j=0}^{k-1}\lambda_{-j}^{2}}, \qquad k\in\mathbb{N}\cap[0,\kappa-1], \\
        \label{FormConditionOnMoments3}
        &\sum_{i=1}^{\eta} \lambda_{i,1}^{2}\int_{0}^{\infty} \frac{1}{t^{\kappa+1}}\ddd\mu_i(t) \le \frac{1}{\prod_{j=0}^{\kappa-1}\lambda_{-j}^{2}},\\
        \label{FormSupremumOfSupremums}
        &\sup_{i\in\mathbb{N}\cap [1,\eta]} \sup\mathrm{supp}\,\mu_i <\infty.
        \end{align}
    where $ \mu_i = \mu_{i,1}^{\pmb{\lambda}'} $ is $ K_i $-atomic, $ i\in\mathbb{N}\cap [1,\eta] $. Set
    \[ s_{i,-k} = \int_{0}^{\infty} \frac{1}{t^{k}}\ddd\mu_i(t), \qquad k\in\mathbb{N}\cap[1,\kappa+1],\ i\in\mathbb{N}\cap[1,\eta]. \]
    Then, for every $ i\in\mathbb{N}\cap [1,\eta] $ the sequence $ (s_{i,-\kappa-1},\ldots,s_{i,p-1}) $ is a moment sequence on $ (0,\infty) $ of index $ K_i $. The application of Theorem \ref{ThmConditionsOnMomentsOfMinimalMeasureInfiniteInterval} gives \eqref{FormSubnormalCompletionStrictlyPositiveMoments} and \eqref{FormSubnormalCompletionSingularlyPositiveMoments}. The formulas \eqref{FormSubnormalCompletionEquality} and \eqref{FormSubnormalCompletionInequality} follow from the definition of $ s_{i,k} $. It is enough to show that \eqref{FormSubnormalCompletionSupremumOfAtoms} holds. Let $ \xi_{i,1},\ldots,\xi_{i,K_i}\subset (0,\infty) $ be atoms of $ \mu_i $ written in the increasing order. Then $ \sup_{i\in\mathbb{N}\cap [1,\eta]}(\sup \supp\mu_i) = \sup_{i\in\mathbb{N}\cap [1,\eta]} \xi_{i,K_i} $. Next, observe that
    \begin{align*}
        \sup_{i\in\mathbb{N}\cap [1,\eta]} \xi_{i,K_i} &\le \sup_{i\in\mathbb{N}\cap [1,\eta]}(\xi_{i,1}+\ldots+\xi_{i,K_i}) \le \sup_{i\in\mathbb{N}\cap [1,\eta]} K_i\xi_{i,K_i}\\ &\le \left\lceil \frac{p+\kappa+1}{2}\right\rceil\sup_{i\in\mathbb{N}\cap [1,\eta]} \xi_{i,K_i} .
    \end{align*}
    Hence,
    \begin{equation}
        \label{FormSupremumOfAtoms}
        \sup_{i\in\mathbb{N}\cap [1,\eta]} \xi_{i,K_i} <\infty \iff \sup_{i\in\mathbb{N}\cap [1,\eta]}(\xi_{i,1}+\ldots+\xi_{i,K_i}) <\infty.
    \end{equation}
    Define $ Q_i(t) = \prod_{j=1}^{K_i}(t-\xi_{i,j}) $. Using Vieta's formula (see \cite{W1}), we get
    \[  \xi_{i,1}+\ldots+\xi_{i,K_i} = -\frac{1}{K_i}\frac{Q_i^{(K_i-1)}(0)}{Q_i^{(K_i)}(0)}, \qquad i\in\mathbb{N}\cap[1,\eta] \]
    But
    \[ -\frac{1}{\left\lceil \frac{p+\kappa+1}{2}\right\rceil}\frac{Q_i^{(K_i-1)}(0)}{Q_i^{(K_i)}(0)} \le-\frac{1}{K_i}\frac{Q_i^{(K_i-1)}(0)}{Q_i^{(K_i)}(0)}\le -\frac{Q_i^{(K_i-1)}(0)}{Q_i^{(K_i)}(0)}, \qquad i\in\mathbb{N}\cap[1,\eta] \]
    so
    \begin{equation}
        \label{FormSupremumOfRoots}
        \sup_{i\in\mathbb{N}\cap [1,\eta]}(\xi_{i,1}+\ldots+\xi_{i,K_i}) <\infty \iff   \sup_{i\in\mathbb{N}\cap [1,\eta]}-\frac{Q_i^{(K_i-1)}(0)}{Q_i^{(K_i)}(0)} < \infty.
    \end{equation}
    (ii)$ \Longrightarrow $(i). It follows from \eqref{FormSubnormalCompletionStrictlyPositiveMoments}, \eqref{FormSubnormalCompletionSingularlyPositiveMoments} and Theorem \ref{ThmExistenceOfBackwardExtensionWithPrescribedNumberOfAtomsInfiniteInterval} that for every $ i\in\mathbb{N}\cap[1,\eta] $ the sequence $ \mathbf{s}_i = (s_{i,-\kappa-1},\ldots,s_{i,p-1}) $ is a positive sequence on $ (0,\infty) $ of index $ K_i $. Let $ \nu_i\in\mathcal{M}_{\infty}(\mathbf{s}_i) $ be $ K_i $-atomic. Set $ \mu_i = t^{\kappa+1}\ddd\nu_i $. It can be easily seen that, by \eqref{FormSubnormalCompletionEquality} and \eqref{FormSubnormalCompletionInequality}, conditions \eqref{FormConditionOnMoments1}, \eqref{FormConditionOnMoments2} and \eqref{FormConditionOnMoments3} are satisfied. Proceeding in the similar way as to obtain \eqref{FormSupremumOfAtoms} and \eqref{FormSupremumOfRoots} we get \eqref{FormSupremumOfSupremums}. Using \cite[Lemma 4.7]{EJSY1} we obtain (i), which completes the proof.
\end{proof}
\begin{remark}
    \label{RemWhyCanLookForFinitelyAtomicMeasuresInSubnormalCompletion}
    Under the assumptions of Theorem \ref{ThmSubnormalCompletion}, if $ \pmb{\lambda} $ has a subnormal completion $ S_{\pmb{\lambda}'} $, then for every $ i\in \mathbb{N}\cap[1,\eta] $ the sequence $ \left( \int_{0}^{\infty} t^{k}\ddd\mu_{i,1}^{\pmb{\lambda'}}(t) \right)_{k=-\kappa-1}^{p-1} $ is a moment sequence on $ (0,\infty) $. By Lemma \ref{LemIndexEstimationInfiniteInterval}, this sequence has always a representing measure, which has at most $ \left\lceil \frac{p+\kappa+1}{2}\right\rceil $ atoms, so after changing weights $ \lambda'_{i,k} $ ($ i\in \mathbb{N}\cap[1,\eta], \ k\in \mathbb{N}_{p+1} $) we can obtain another subnormal completion $ S_{\pmb{\lambda}''} $ such that all measures $ \mu_{i,1}^{\pmb{\lambda''}} $ ($ i\in \mathbb{N}\cap[1,\eta] $) are finitely atomic with at most $ \left\lceil \frac{p+\kappa+1}{2}\right\rceil $ atoms. Therefore, Theorem \ref{ThmSubnormalCompletion} provides a full description of sequences $ \pmb{\lambda} $ having subnormal completion.
\end{remark}
In the next example we show that \cite[Theorem 4.1]{EJSY2} can be derived from Theorem \ref{ThmSubnormalCompletion}.
\begin{example}
    We investigate the case $ \kappa = 1 $, $ \eta\in\overline{\mathbb{N}}_{2} $, $ p = 2 $ and $ K_i = 2 $ for $ i\in\mathbb{N}\cap[1,\eta] $. Since for any $ t\in (0,\infty) $ the sequence $ (1,t) $ is a moment sequence on $ (0,\infty) $, we have to check only \eqref{FormSubnormalCompletionEquality} and \eqref{FormSubnormalCompletionInequality}. Note that, in our setting, the condition \eqref{FormSubnormalCompletionSingularlyPositiveMoments} does not hold for any $ i\in\mathbb{N}\cap[1,\eta] $. First, we compute $ t_{\infty}((1,\lambda_{i,2}^{2})) $. From \cite[Theorem IV.1.1]{KN1}, \cite[Section III.\S 5.3]{KN1} and \eqref{FormBackwardExtensionCompactIntervalInfimum}--\eqref{FormBackwardExtensionCompactIntervalSupremum} we know that $ t_{\infty}((1,\lambda_{i,2}^{2})) = t_{a,b}((1,\lambda_{i,2}^{2})) = -\frac{\underline{P}_{i}(0)}{\underline{Q}_{i}(0)} $, where
    \begin{align*}
        \underline{Q}_{i}(t) &= \det\begin{bmatrix}
        1 & 1\\
        \lambda_{i,2}^{2} & t
        \end{bmatrix}, \qquad  i\in\mathbb{N}\cap[1,\eta],\\
        \underline{P}_{i}(0) &= \sigma\left(\frac{\underline{Q}_{i}(t)-\underline{Q}_{i}(0)}{t}\right),\qquad i\in\mathbb{N}\cap[1,\eta].
    \end{align*}
    By simple calculations we get $ t_{\infty}((1,\lambda_{i,2}^{2})) = \frac{1}{\lambda_{i,2}^{2}} $, $ i\in\mathbb{N}\cap[1,\eta] $. Next, assuming $ s_{i,-1}\in (\frac{1}{\lambda_{i,2}^{2}},\infty) $ is chosen, we compute $ t_{\infty}((s_{i,-1},1,\lambda_{i,2}^{2})) $. Again, from \cite[Theorem IV.1.1]{KN1}, \cite[Section III.\S 5.3]{KN1} and \eqref{FormBackwardExtensionCompactIntervalInfimum}--\eqref{FormBackwardExtensionCompactIntervalSupremum}, it follows that $ t_{a,b}(((s_{i,-1},1,\lambda_{i,2}^{2}))) $ depends on $ b $ (but not on $ a $). Hence, by the fact that $ t_{a,b}(\,\cdot\,) $ decreases when $ b $ increases, we have that for every $ i\in\mathbb{N}\cap[1,\eta] $,
    \begin{equation*}
        t_{\infty}((s_{i,-1},1,\lambda_{i,2}^{2})) = \lim_{b\to\infty} t_{a,b}((s_{i,-1},1,\lambda_{i,2}^{2})) = \lim_{b\to\infty} - \frac{\overline{P}_{i}^{a,b}(0)}{\overline{Q}_{i}^{a,b}(0)},
    \end{equation*}
    where
    \begin{align*}
        \overline{Q}_{i}^{a,b}(t) &= (b-t)\det\begin{bmatrix}
        bs_{i,-1}-1 & 1\\
        b-\lambda_{i,2}^{2} & t
        \end{bmatrix}, \qquad i\in\mathbb{N}\cap[1,\eta],\ a,b\in(0,\infty),\ a<b\\
        \overline{P}_{i}^{a,b}(0) &= \sigma\left(\frac{\overline{Q}_{i}^{a,b}(t)-\overline{Q}_{i}^{a,b}(0)}{t}\right), \qquad i\in\mathbb{N}\cap[1,\eta],\ a,b\in(0,\infty),\ a<b.
    \end{align*}
    By simple calculations we get
    \begin{align*}
        t_{\infty}((s_{i,-1},1,\lambda_{i,2}^{2})) &= \lim_{b\to\infty} \frac{b^{2}s_{i,-1}^{2}-bs_{i,-1}+s_{i,-1}+s_{i,-1}\lambda_{i,2}^{2}}{b^{2}-b\lambda_{i,2}^{2}} \\
        &= s_{i,-1}^{2} > 0, \qquad i\in\mathbb{N}\cap[1,\eta].
    \end{align*}
    Hence, $ s_{i,-1}\in (\frac{1}{\lambda_{i,2}^{2}},\infty) $ and $ s_{i,-2}\in (s_{i,-1}^{2},\infty) $. Set $ r_{i} = \lambda_{i,2}^{2}s_{i,-1} \in (1,\infty) $ and $ \vartheta_{i} = \frac{s_{i,-2}}{s_{i,-1}^{2}} \in (1,\infty) $, $ i\in\mathbb{N} \cap[1,\eta] $. Then \eqref{FormSubnormalCompletionEquality} and \eqref{FormSubnormalCompletionInequality} take the form
    \begin{align*}
        &\sum_{i=1}^{\eta} \frac{\lambda_{i,1}^{2}}{\lambda_{i,2}^{2}}r_{i} = 1\\
        &\sum_{i=1}^{\eta} \frac{\lambda_{i,1}^{2}}{\lambda_{i,2}^{4}}\vartheta_{i}r_{i}^{2} \le \frac{1}{\lambda_{0}^{2}},
    \end{align*}
    as in \cite[Eq. (4.2) and (4.3)]{EJSY2}.
    From Theorem \ref{ThmExistenceOfBackwardExtensionWithPrescribedNumberOfAtomsInfiniteInterval} we get that the polynomials $ Q_{i} $ in Theorem \ref{ThmSubnormalCompletion} are given by the formula:
    \begin{align*}
        &Q_i(t) = \det\begin{bmatrix}
        s_{i,-2} & s_{i,-1} & 1\\
        s_{i,-1} & 1 & t\\
        1 & \lambda_{i,2}^{2} & t^{2}
        \end{bmatrix}, \qquad i\in\mathbb{N}\cap[1,\eta].
    \end{align*} 
    Then, we have
    \begin{align*}
        &Q_{i}'(0) = \det\begin{bmatrix}
        s_{i,-2} & s_{i,-1} & 0\\
        s_{i,-1} & 1 & 1\\
        1 & \lambda_{i,2}^{2} & 0
        \end{bmatrix} = -\lambda_{i,2}^{2}s_{i,-2}+s_{i,-1}, \qquad i\in\mathbb{N}\cap[1,\eta].
    \end{align*}
    and
    \begin{align*}
        &Q_{i}''(0) = \det\begin{bmatrix}
        s_{i,-2} & s_{i,-1} & 0\\
        s_{i,-1} & 1 & 0\\
        1 & \lambda_{i,2}^{2} & 1
        \end{bmatrix} = s_{i,-2}-s_{i,-1}^{2}, \qquad i\in\mathbb{N}\cap[1,\eta].
    \end{align*}
    Therefore,
    \begin{align*}
        -\frac{Q_{i}'(0)}{Q_{i}''(0)} &= \frac{\frac{r_{i}^{2}\vartheta_{i}}{\lambda_{i,2}^{2}}-\frac{r_{i}}{\lambda_{i,2}^{2}}}{\frac{r_{i}^{2}\vartheta_{i}}{\lambda_{i,2}^{4}}} = \lambda_{i,2}^{2} \frac{\vartheta_{i}r_{i}-1}{r_{i}(\vartheta_i-1)}, \qquad i\in\mathbb{N}\cap[1,\eta].
    \end{align*}
    Hence, by the above equality, \eqref{FormSubnormalCompletionSupremumOfAtoms} is equivalent to \cite[Eq. (4.1)]{EJSY2}. Consequently, \cite[Theorem 4.1]{EJSY2} can be recovered from our result.
\end{example}
Let us make one more remark about \eqref{FormSubnormalCompletionSupremumOfAtoms}. Obviously, this condition matters only when $ \eta = \infty $. When we take $ K_i \le \frac{p+\kappa+1}{2} $, then polynomial $ Q_i $ is given as in Theorem \ref{ThmExistenceOfBackwardExtensionWithPrescribedNumberOfAtomsInfiniteInterval}. However, when $ p+\kappa $ is even, then it can occur $ K_i = \frac{p+\kappa+2}{2} $, so that Theorem \ref{ThmExistenceOfBackwardExtensionWithPrescribedNumberOfAtomsInfiniteInterval} gives no help in finding appropriate polynomial $ Q_{i} $. Of course, if there are only finitely many indices $ i\in\mathbb{N}\cap [1,\eta] $, for which $ K_i = \frac{p+\kappa+2}{2} $, then we can simply skip these indices when checking \eqref{FormSubnormalCompletionSupremumOfAtoms}. The problem appears when there are infinitely many such indices. Observe that 
\[ \sup_{i\in\mathbb{N}\cap [1,\eta]} -\frac{Q_{i}^{(K_i-1)}(0)}{Q_{i}^{(K_i)}(0)} = \max \left\{\sup_{i\in I_1} -\frac{Q_{i}^{(K_i-1)}(0)}{Q_{i}^{(K_i)}(0)}, \sup_{i\in I_2} -\frac{Q_{i}^{(K_i-1)}(0)}{Q_{i}^{(K_i)}(0)} \right\}, \]
where $ I_1 = \{i\in\mathbb{N}\cap[1,\eta]\!: K_i \le \frac{p+\kappa+1}{2} \} $ and $ I_2 = \{i\in\mathbb{N}\cap[1,\eta]\!: K_i = \frac{p+\kappa+2}{2} \} $. For $ i\in I_2 $ set $ \mathbf{s}_{i} := (s_{i,-\kappa-1},\ldots,s_{i,p-1}) $ and denote by $ \Phi_{i} $ the bijection given by Corollary \ref{CorParametrizationOfMinimalMeasuresOfStrictlyPositiveSequenceEvenInfiniteInterval} for the sequence $ \mathbf{s}_{i} $. If $ i\in I_{2} $, then by $ Q_{i,t} $ denote the polynomial with roots at atoms of the measure $ \Phi_{i}^{-1}(t) $, $ t\in(t_{\infty}(\mathbf{s}_{i}),\infty) $. By Corollary \ref{CorParametrizationOfMinimalMeasuresOfStrictlyPositiveSequenceEvenInfiniteInterval}, the polynomials $ Q_{i} $, $ i\in I_2 $, are of the form $ Q_i = Q_{i,s_{i,-\kappa-2}} $ for some $ s_{i,-\kappa-2}\in (t_{\infty}(\mathbf{s}_{i}),\infty) $. It is easy to see that $ \sup_{i\in\mathbb{N}\cap [1,\eta]} -\frac{Q_{i}^{(K_i-1)}(0)}{Q_{i}^{(K_i)}(0)} < \infty $ if and only if $ \max\{S_1, S_2 \} < \infty $, where
\begin{align*}
    S_1 &= \sup_{i\in I_1} -\frac{Q_{i}^{(K_i-1)}(0)}{Q_{i}^{(K_i)}(0)}\\
    S_2 &= \sup_{i\in I_2} \inf\left\{ -\frac{(Q_{i,s_{i,-\kappa-2}})^{(K_i-1)}(0)}{Q_{i}^{(K_i)}(0)}\!: s_{i,-\kappa-2}\in (t_{\infty}(\mathbf{s}_{i}),\infty)\right\},
\end{align*}
which gives us the convenient way of checking \eqref{FormSubnormalCompletionSupremumOfAtoms} in any case.\par
Let $ \kappa \in\mathbb{N} $, $ \eta\in\overline{\mathbb{N}}_{2} $, $ p\in\mathbb{N}_{1} $. Assume $ \pmb{\lambda} = \{\lambda_{-k}\}_{k=0}^{\kappa-1}\cup \{\lambda_{i,j} \}_{i,j=1}^{\eta,p} \subset (0,\infty) $ admits a subnormal completion. For $ M \in (0,\infty) $ denote by $ \mathcal{U}_{\pmb{\lambda}}^{M} $ the set of all sequences $ (\mu_{i})_{i=1}^{\eta} $ of Borel measures on $ [0,M] $ satisfying
\begin{align}
    \label{FormMeasuresInSubnormalCompletionCondition1}
    &\prod_{j=2}^{n+1}\lambda_{i,j}^{2} = \int_{0}^{M}t^{n}\ddd\mu_i(t), \qquad i\in\mathbb{N}\cap [1,\eta], \ n\in\mathbb{N}\cap[0,p-1],\\
    \label{FormMeasuresInSubnormalCompletionCondition2}
    &\sum_{i=1}^{\eta}\lambda_{i,1}^{2}\int_{0}^{M}\frac{1}{t^{k+1}}\ddd\mu_i(t) = \frac{1}{\prod_{j=0}^{k-1}\lambda_{-j}^{2}}, \qquad k\in\mathbb{N}\cap[0,\kappa-1],\\
    \label{FormMeasuresInSubnormalCompletionCondition3}
    &\sum_{i=1}^{\eta}\lambda_{i,1}^{2}\int_{0}^{M}\frac{1}{t^{\kappa+1}}\ddd\mu_i(t) \le \frac{1}{\prod_{j=0}^{\kappa-1}\lambda_{-j}^{2}}.
\end{align}
Obviously, $ \mathcal{U}_{\pmb{\lambda}}^{M} \subset \mathcal{P}([0,M])^{\eta} $, where $ \mathcal{P}([0,M]) $ stands for the set of all Borel probability measures on $ [0,M] $. By Banach-Alaoglu and Riesz representation theorems we see that $ \mathcal{P}([0,M]) $ equipped with the weak topology is a compact space, hence, by Tychonoff's theorem, so is $ \mathcal{P}([0,M])^{\eta} $; moreover, it is metrizable (see \cite{B1}). We will prove that $ \mathcal{U}_{\pmb{\lambda}}^{M} $ is a compact subset of $ \mathcal{P}([0,M])^{\eta} $. Before that, we need the following lemmas. For the sake of completeness we include the proof of Lemma \ref{LemReciprocalIntegrableLimit}; for the proof of Lemma \ref{LemConvergenceOfIntegralOfReciprocal} consult \cite[Theorem I.5.3]{B1}.
\begin{lemma}
    \label{LemReciprocalIntegrableLimit}
    Suppose $ M\in (0,\infty) $ and $ (\mu_n)_{n=0}^{\infty}\subset \mathcal{P}([0,M]) $ is such that $ \mu_n \to \mu \in \mathcal{P}([0,M]) $ weakly. Let $ k\in\mathbb{N}_1 $ and assume that 
    \[ C:= \sup_{n\in\mathbb{N}} \int_{0}^{M} \frac{1}{t^{k}}\ddd\mu_n(t) < \infty. \]
    Then $ \frac{1}{t^{k}}\in L^{1}(\mu) $ and $ \int_{0}^{M}\frac{1}{t^{k}}\ddd\mu(t) \le C $.
\end{lemma}
\begin{proof}
    For $ m\in\mathbb{N}_1 $ define a (continuous) function $ f_m\!: [0,M]\to [0,\infty) $ by the formula\footnote{We stick to the convention that $ \frac{1}{0} = \infty $}
    \begin{equation*}
        f_{m}(t) = \min\{\frac{1}{t^{k}}, m^{k} \}, \qquad t\in [0,M].
    \end{equation*}
    Observe that $ f_{m}\le f_{m+1} $, $ m\in\mathbb{N}_1 $, and $ f_{m} \stackrel{m\to\infty}{\longrightarrow} \frac{1}{t^{k}} $ pointwise. Since
    \begin{equation*}
        \int_{0}^{M} f_{m}\ddd\mu_{n} \le \int_{0}^{M} \frac{1}{t^{k}}\ddd\mu_n(t)\le C, \qquad m\in\mathbb{N}_1,
    \end{equation*}
    and
    \begin{equation*}
        \int_{0}^{M}f_{m}\ddd\mu_n \stackrel{n\to \infty}{\longrightarrow} \int_{0}^{M} f_{m}\ddd\mu,
    \end{equation*}
    it follows that $ \int_{0}^{M} f_{m}\ddd\mu \le C $. By the Lebesgue's monotone convergence theorem we obtain
    \begin{equation*}
        \int_{0}^{M} f_m\ddd\mu \stackrel{m\to \infty}{\longrightarrow} \int_{0}^{M} \frac{1}{t^{k}}\ddd\mu(t).
    \end{equation*}
    Hence, $ \int_{0}^{M} \frac{1}{t^{k}}\ddd\mu(t) \le C $.
\end{proof}
\begin{lemma}
    \label{LemConvergenceOfIntegralOfReciprocal}
    Let $ M \in (0,\infty) $, $ k\in\mathbb{N}_{1} $. Let $ (\mu_n)_{n=0}^{\infty}\subset \mathcal{P}([0,M]) $ be such that $ \mu_n \to \mu \in \mathcal{P}([0,M]) $ weakly. Assume that
    \begin{equation*}
        \sup_{n\in\mathbb{N}} \int_{0}^{M} \frac{1}{t^{k+1}}\ddd\mu_n(t) < \infty.
    \end{equation*}
    Then $ \frac{1}{t^{k}}\in L^{1}(\mu) $ and $ \int_{0}^{M}\frac{1}{t^{k}}\ddd\mu_n(t) \to \int_{0}^{M}\frac{1}{t^{k}}\ddd\mu(t) $.
\end{lemma}
Now we are in the position to prove that the sets $ \mathcal{U}_{\pmb{\lambda}}^{M} $ are compact in the weak topology.
\begin{theorem}
    \label{ThmCompactnessOfSetOfMeasuresWithGivenSupport}
    Let $ \kappa \in\mathbb{N} $, $ \eta\in\overline{\mathbb{N}}_{2} $, $ p\in\mathbb{N}_{1} $. Assume $ \pmb{\lambda} = \{\lambda_{-k}\}_{k=0}^{\kappa-1}\cup \{\lambda_{i,j} \}_{i,j=1}^{\eta,p} \subset (0,\infty) $ admits a subnormal completion. Then for every $ M\in (0,\infty) $ the set $ \mathcal{U}_{\pmb{\lambda}}^{M} $ is weakly compact.
\end{theorem}
\begin{proof}
    Fix $ M \in (0,\infty) $. Since $ \mathcal{P}([0,M])^{\eta} $ is weakly compact, it is enough to prove that $ \mathcal{U}_{\pmb{\lambda}}^{M} $ is weakly closed in $ \mathcal{P}([0,M])^{\eta} $. Let $ ((\mu_{i}^{(n)})_{i=1}^{\eta})_{n=0}^{\infty}\subset \mathcal{U}_{\pmb{\lambda}}^{M} $ be such that $ (\mu_{i}^{(n)})_{i=1}^{\eta} \to (\mu_{i})_{i=1}^{\eta}\in \mathcal{P}([0,M])^{\eta} $ weakly. It is easy to see that \eqref{FormMeasuresInSubnormalCompletionCondition1} holds for $ (\mu_{i})_{i=1}^{\eta} $. For $ n\in\mathbb{N} $ define a measure $ \tau_{n}\in\mathcal{P}([0,M]) $:
    \begin{equation*}
    \tau_{n}(A) = \frac{1}{\sum_{i=1}^{\eta}\lambda_{i,1}^{2}}\sum_{i=1}^{\eta}\lambda_{i,1}^{2}\mu_{i}^{(n)}(A), \qquad A\in\mathcal{B}([0,M]),
    \end{equation*}
    and a measure $ \tau\in\mathcal{P}([0,M]) $:
    \begin{equation*}
    \tau(A) = \frac{1}{\sum_{i=1}^{\eta}\lambda_{i,1}^{2}}\sum_{i=1}^{\eta}\lambda_{i,1}^{2}\mu_{i}(A), \qquad A\in\mathcal{B}([0,M]). 
    \end{equation*}
    We will prove that $ \tau_{n}\to \tau $ weakly. Suppose $ f\!:[0,M]\to \mathbb{R} $ is continuous. Note that
    \[ \left\lvert\lambda_{i,1}^{2}\int_{0}^{M}f\ddd\mu_i^{(n)}\right\rvert \le \lambda_{i,1}^{2}\sup\lvert f\rvert([0,M]), \qquad n\in\mathbb{N}, \ i\in\mathbb{N}\cap[1,\eta] \]
    and $ \sum_{i=1}^{\eta} \lambda_{i,1}^{2}\sup\lvert f\rvert([0,M]) < \infty $. Since
    \begin{equation*}
    \int_{0}^{M}f\ddd\mu_i^{(n)} \stackrel{n\to\infty}{\longrightarrow} \int_{0}^{M}f\ddd\mu_{i}, \qquad i\in\mathbb{N}\cap[1,\eta],
    \end{equation*}
    by the Lebesgue's dominated convergence theorem we have
    \begin{align*}
    \int_{0}^{M}f\ddd\tau_n &= \frac{1}{\sum_{i=1}^{\eta}\lambda_{i,1}^{2}}\sum_{i=1}^{\eta}\lambda_{i,1}^{2}\int_{0}^{M}f\ddd\mu_{i}^{(n)} \\
    &\stackrel{n\to\infty}{\longrightarrow} \frac{1}{\sum_{i=1}^{\eta}\lambda_{i,1}^{2}}\sum_{i=1}^{\eta}\lambda_{i,1}^{2}\int_{0}^{M}f\ddd\mu_{i} = \int_{0}^{M}f\ddd\tau.
    \end{align*}
    From \eqref{FormMeasuresInSubnormalCompletionCondition2} and \eqref{FormMeasuresInSubnormalCompletionCondition3}, by Lemma \ref{LemReciprocalIntegrableLimit}, it follows that
    \begin{equation*}
        \sum_{i=1}^{\eta}\lambda_{i,1}^{2}\int_{0}^{M} \frac{1}{t^{m+1}} \ddd\tau(t) \le \frac{1}{\prod_{j=0}^{m-1}\lambda_{-j}^{2}}, \qquad m\in\mathbb{N}\cap [0,\kappa].
    \end{equation*} 
    Now, we can apply Lemma \ref{LemConvergenceOfIntegralOfReciprocal} to obtain that for every $ m\in\mathbb{N}\cap[0,\kappa-1] $ the following equality holds:
    \begin{equation*}
        \sum_{i=1}^{\eta}\lambda_{i,1}^{2}\int_{0}^{M}\frac{1}{t^{m+1}}\ddd\mu_i = \sum_{i=1}^{\eta}\lambda_{i,1}^{2}\int_{0}^{M} \frac{1}{t^{m+1}}\ddd\tau(t) = \frac{1}{\prod_{j=0}^{m-1}\lambda_{-j}^{2}}.
    \end{equation*}
    This completes the proof.
\end{proof}
As one can observe, Theorem \ref{ThmSubnormalCompletion} gives the solution for the subnormal completion problem only in case $ \kappa\in\mathbb{N} $. The case $ \kappa = \infty $ remains unsolved. In \cite[Problem 4.10]{EJSY2} there was posed a question whether existence of subnormal completion for every $ \kappa\in\mathbb{N} $ implies existence of such a completion for $ \kappa = \infty $. The next theorem gives a partial affirmative answer to this question.
\begin{theorem}
    \label{ThmSubnormalCompletionKappaInfinite}
    Let $ \eta\in\overline{\mathbb{N}}_{2} $ and $ p\in\mathbb{N}_{1} $. Let $ \pmb{\lambda} = \{\lambda_{-j} \}_{j=0}^{\infty}\cup\{\lambda_{i,j} \}_{i,j=1}^{\eta,p}\subset (0,\infty) $. Set $ \pmb{\lambda}_{\kappa} = \{\lambda_{-j} \}_{j=0}^{\kappa-1}\cup\{\lambda_{i,j} \}_{i,j=1}^{\eta,p} $, $ \kappa\in\mathbb{N}_{1} $, and $ \pmb{\lambda}_{0} = \{\lambda_{i,j} \}_{i,j=1}^{\eta,p} $. The following conditions are equivalent:
    \begin{enumerate}
        \item $ \pmb{\lambda} $ has a subnormal completion,
        \item there exists a sequence $ (S_{\pmb{\lambda}_{\kappa}'})_{\kappa=0}^{\infty} $ of weighted shifts such that $ S_{\pmb{\lambda}_{\kappa}'} $ is a subnormal completion of $ \pmb{\lambda}_{\kappa} $ for every $ \kappa\in\mathbb{N} $ and $ \sup_{\kappa\in\mathbb{N}} \lVert S_{\pmb{\lambda}_{\kappa}'}\rVert < \infty $.
    \end{enumerate}
\end{theorem}
\begin{proof}
    (i)$ \Longrightarrow $(ii). If $ S_{\pmb{\lambda}'} $ is a subnormal completion of $ \pmb{\lambda} $, then $ S_{\pmb{\lambda}'}|_{\ell^{2}(V_{\eta,\kappa})} $ is a subnormal completion of $ \pmb{\lambda}_{\kappa} $ for every $ \kappa\in\mathbb{N} $; moreover $ \sup_{\kappa\in\mathbb{N}}\lVert S_{\pmb{\lambda}'}|_{\ell^{2}(V_{\eta,\kappa})}\rVert \le \lVert S_{\pmb{\lambda}'}\rVert $.\\
    (ii)$ \Longrightarrow $(i). From (ii) and \cite[Lemma 4.7]{EJSY1} we know that $ \mathcal{U}_{\pmb{\lambda}_{\kappa}}^{M}\not=\varnothing $ for every $ \kappa\in\mathbb{N} $, where $ M = \sup_{\kappa\in\mathbb{N}} \lVert S_{\pmb{\lambda}_{\kappa}'}\rVert $. By Theorem \ref{ThmCompactnessOfSetOfMeasuresWithGivenSupport}, $ \mathcal{U}_{\pmb{\lambda}_{\kappa}}^{M} $ is compact. Since $ \mathcal{U}_{\pmb{\lambda}_{\kappa+1}}^{M} \subset \mathcal{U}_{\pmb{\lambda}_{\kappa}}^{M} $, by Cantor's intersection theorem we obtain that
    \begin{equation*}
        \bigcap_{\kappa=0}^{\infty} \mathcal{U}_{\pmb{\lambda}_{\kappa}}^{M}\not= \varnothing.
    \end{equation*}
    Let $ (\mu_{i})_{i=1}^{\eta} \in \bigcap_{\kappa=0}^{\infty} \mathcal{U}_{\pmb{\lambda}_{\kappa}}^{M} $. Then
    \begin{equation*}
        \sum_{i=1}^{\eta}\lambda_{i,1}^{2}\int_{0}^{M}\frac{1}{t^{\kappa+1}}\ddd\mu_{i}(t) = \frac{1}{\prod_{j=0}^{\kappa-1}\lambda_{-j}^{2}}, \qquad \kappa\in\mathbb{N}.
    \end{equation*}
    Setting
    \begin{align}
        &\lambda_{-j}' = \lambda_{-j}, \qquad j\in\mathbb{N},\\
        &\lambda_{i,j}' = \lambda_{i,j}, \qquad i\in\mathbb{N}\cap[1,\eta], \ j\in\mathbb{N}\cap[1,p],\\
        &\lambda_{i,j}' = \sqrt{\frac{\int_{0}^{M}t^{j-1}\ddd\mu_{i}(t)}{\int_{0}^{M}t^{j-2}\ddd\mu_{i}(t)}}, \qquad i\in\mathbb{N}\cap[1,\eta], \ j\in\mathbb{N}_{p+1},
    \end{align}
    and using \cite[Corollary 6.2.2]{JJS1} (as well as \cite[Procedure 6.3.1]{JJS1}) we obtain a subnormal completion $ S_{\pmb{\lambda}'} $ of $ \pmb{\lambda} $. This completes the proof.
\end{proof}
In the hypothesis of \cite[Problem 4.10]{EJSY2} there is no assumption of uniform boundedness as in Theorem \ref{ThmSubnormalCompletionKappaInfinite}.(ii). At this point, we do not know whether this assumption is superfluous or not, but it leads us to another natural problem.
\begin{problem}
    Let $ \kappa\in\mathbb{N}, \ \eta\in\overline{\mathbb{N}}_{2}, \ p\in\mathbb{N}_{1} $. Assume $ \pmb{\lambda} = \{\lambda_{-j} \}_{j=0}^{\kappa-1}\cup\{\lambda_{i,j} \}_{i,j=1}^{\eta,p} \subset (0,\infty) $ admits a subnormal completion. Compute
    \begin{equation*}
        \inf\{\lVert S_{\pmb{\lambda}'}\rVert\!: S_{\pmb{\lambda}'} \text{ is a subnormal completion of } \pmb{\lambda} \}.
    \end{equation*}
\end{problem}

	\section{Completely hyperexpansive completions of weighted shifts on directed trees.}
\label{SecCompletelyHyperexpansiveCompletions}
It is the continuation of considerations from Section \ref{SecSubnormalCompletions}, but now we are interested in completely hyperexpansive completions. Several partial results in this area can be found in \cite{Lee1}. Recall from \cite[Remark 1]{Ath1} that the sequence $ (c_{n})_{n=0}^{\infty}\subset \mathbb{R} $ is completely alternating if and only if there exists a Borel measure $ \tau\!: \mathcal{B}([0,1])\to[0,\infty) $ such that
\begin{equation}
    \label{FormCompletelyAlternatingAsMoments}
    c_{n} = c_{0} + \int_{[0,1]} (1+t+\ldots+t^{n-1})\ddd\tau(t), \qquad n\in \mathbb{N}_{1}.
\end{equation}
Moreover, it can be easily seen that the measure $ \tau $ satisfying \eqref{FormCompletelyAlternatingAsMoments} is unique; we call it \textit{a representing measure of} $ (c_{n})_{n=0}^{\infty} $. If $ \mathcal{T} = (V,E) $ is a directed tree, then \cite[Lemma 7.1.4]{JJS1} states that a bounded weighted shift $ S_{\pmb{\lambda}}\in \ell^{2}(V) $ is completely hyperexpansive if and only if for every $ v\in V $ the sequence $ (\lVert S_{\pmb{\lambda}}^{n}e_{v}\rVert^{2})_{n=0}^{\infty} $ is completely alternating; the representing measure for the sequence $ (\lVert S_{\pmb{\lambda}}^{n}e_{v}\rVert^{2})_{n=0}^{\infty} $ will be denoted by $ \tau_{v}^{\pmb{\lambda}} $. 
Exploiting once again our results on backward extensions we obtain the following counterpart of Theorem \ref{ThmSubnormalCompletion}.
\begin{theorem}
    \label{ThmCompletelyHyperexpansiveCompletion}
    Let $ \kappa\in\mathbb{N} $, $ \eta\in\overline{\mathbb{N}}_{2} $, $ p\in\mathbb{N}_{2} $. Assume $ \pmb{\lambda} = \{\lambda_{-k} \}_{k=0}^{\kappa-1}\cup\{\lambda_{i,j} \}_{i=1,j=1}^{\eta,p} \subset (0,\infty) $ is such that $ \lambda_{i,j} > 1 $ for $ i\in\mathbb{N}\cap[1,\eta], \ j\in\mathbb{N}\cap[2,p] $ and
    \[ \sum_{i=1}^{\eta}\lambda_{i,1}^{2} <\infty. \] Suppose $ \{K_{i}\}_{1\le i\le \eta} \subset \mathbb{Q}\cap \left[\frac{1}{2},\frac{p+\kappa}{2}\right] $ satisfies $ 2K_{i}\in\mathbb{N} $, $ i\in \mathbb{N}\cap [1,\eta] $. Then the following conditions are equivalent:
    \begin{enumerate}
        \item there exists a completely hyperexpansive completion $ S_{\pmb{\lambda'}}\in \mathbf{B}(\ell^{2}(V_{\eta,\kappa})) $ of $ \pmb{\lambda} $ such that $ \tau_{i,1}^{\pmb{\lambda}'} $ is of index $ K_i $ for $ i\in\mathbb{N}\cap[1,\eta] $,
        \item for every $ i\in\mathbb{N}\cap[1,\eta] $ there exists a sequence $ \{s_{i,-k}\}_{k=1}^{\kappa+1}\subset (0,\infty) $ such that for every $ i\in\mathbb{N}\cap[1,\eta] $
        \begin{equation}
            \label{FormCompletelyHyperexpansiveStrictlyPositiveMoments}
            s_{i,k} \in (t_{1}((s_{i,k+1},\ldots,s_{i,p-1})),+\infty), \quad k\in\mathbb{Z}\cap[p-N_i-2,p-3],
        \end{equation}
        \begin{equation}
            \label{FormCompletelyHyperexpansiveSingularlyPositiveMoments}
            s_{i,k} = t_{1}((s_{i,k+1},\ldots,s_{k+N_i+1})), \quad k\in\mathbb{Z}\cap[-\kappa-1,p-N_i-3],
        \end{equation}
        \begin{equation}
            \label{FormCompletelyHyperexpansiveSupremumOfMasses}
            \sup_{i\in\mathbb{N}\cap[1,\eta]} s_{i,0} < \infty
        \end{equation}
        and\footnote{Recall the convention: $ \prod\varnothing = 1 $.}
        \begin{enumerate}
            \item if $ \kappa = 0 $:
            \begin{equation}
                \label{FormCompletelyHyperexpansiveInequalityKappaZero}
                1+\sum_{i=1}^{\eta} \lambda_{i,1}^{2}s_{i,-1} \le \sum_{i=1}^{\eta} \lambda_{i,1}^{2}, 
            \end{equation}
            \item if $ \kappa > 0 $:
            \begin{equation}
                \label{FormCompletelyHyperexpansiveEquality1KappaPositive}
                1+\sum_{i=1}^{\eta} \lambda_{i,1}^{2}s_{i,-1} = \sum_{i=1}^{\eta} \lambda_{i,1}^{2},
            \end{equation}
            \begin{equation}
                \label{FormCompletelyHyperexpansiveEquality2KappaPositive}
                1+\prod_{j=0}^{k-1}\lambda_{-j}^{2}\sum_{i=1}^{\eta} \lambda_{i,1}^{2}s_{i,-k-1} = \lambda_{-k+1}^{2}, \qquad k\in\mathbb{N}\cap[1,\kappa-1],
            \end{equation}
            \begin{equation}
                \label{FormCompletelyHyperexpansiveInequalityKappaPositive}
                1+\prod_{j=0}^{\kappa-1}\lambda_{-j}^{2}\sum_{i=1}^{\eta} \lambda_{i,1}^{2}s_{i,-\kappa-1} \le \lambda_{-\kappa+1}^{2},
            \end{equation}
        \end{enumerate}
        where
        \begin{equation*}
            s_{i,k} = \prod_{j=2}^{k+2}\lambda_{i,j}^{2}- \prod_{j=2}^{k+1}\lambda_{i,j}^{2}, \qquad k\in \mathbb{N}\cap [0,p-2].
        \end{equation*}
    \end{enumerate}
\end{theorem}
\begin{proof}
    (i)$ \Longrightarrow $(ii). By (i) and \cite[Proposition 2.8]{Lee1} there exist Borel measures $ \{\tau_{i}\}_{i=1}^{\eta} $ on $ (0,1] $ such that $ \tau_{i,1}^{\pmb{\lambda'}} = \tau_i $ is of index $ K_i $ for $ i\in\mathbb{N}\cap[1,\eta] $ and (\ref{FormCompletelyHyperexpansiveSupremumOfMasses})--(\ref{FormCompletelyHyperexpansiveInequalityKappaPositive}) hold with
    \[ s_{i,-k} = \int_{(0,1]} \frac{1}{t^{k}}\ddd\tau_{i}(t), \qquad i\in\mathbb{N}\cap[1,\eta],\ k\in\mathbb{N}\cap[0,\kappa+1]. \]
    Then for every $ i\in\mathbb{N}\cap[1,\eta] $ the sequence $ (s_{i,-\kappa-1},\ldots,s_{i,p-2}) $ is a moment sequence on $ (0,1] $ of index $ K_i $, so Theorem \ref{ThmConditionsOnMomentsOfMinimalMeasureHalfOpenInterval} gives (\ref{FormCompletelyHyperexpansiveStrictlyPositiveMoments}) and (\ref{FormCompletelyHyperexpansiveSingularlyPositiveMoments}).\\
    (ii)$ \Longrightarrow $(i). From (\ref{FormCompletelyHyperexpansiveStrictlyPositiveMoments}) and (\ref{FormCompletelyHyperexpansiveSingularlyPositiveMoments}), by Theorem \ref{ThmExistenceOfBackwardExtensionWithPrescribedNumberOfAtomsHalfOpenInterval}, it follows that for every $ i\in\mathbb{N}\cap[1,\eta] $ the sequence $ \mathbf{s}_i = (s_{i,-\kappa-1},\ldots,s_{i,p-2}) $ is a positive sequence on $ (0,1] $ satisfying $ \text{ind}_{1}(\mathbf{s}_i) = K_i $ . Let $ \nu_i\in\mathcal{M}_{1}(\mathbf{s}_i) $ be such that $ \ind_{1}(\nu_{i}) = K_i $, $ i\in\mathbb{N}\cap[1,\eta] $. Set $ \tau_{i} = t^{\kappa+1}\ddd\nu_i $, $ i\in\mathbb{N}\cap[1,\eta] $. Then for every $ i\in\mathbb{N}\cap[1,\eta] $ the measure $ \tau_{i} $ also satisfies $ \text{ind}_{1}(\tau_i) = K_{i} $. By \cite[Proposition 2.8]{Lee1}, $ \pmb{\lambda} $ has a completely hyperexpansive completion $ S_{\pmb{\lambda}'} $ such that $ \tau_{i,1}^{\pmb{\lambda}'} = \tau_i $ is of index $ K_i $ for every $ i\in\mathbb{N}\cap[1,\eta] $.
\end{proof}
Again, when solving completely hyperexpansive completion problem we can always look for measures $ \tau_{i,1}^{\pmb{\lambda}'} $ ($ i\in \mathbb{N}\cap[1,\eta] $) of index at most $ \frac{p+\kappa}{2} $ (cf. Remark \ref{RemWhyCanLookForFinitelyAtomicMeasuresInSubnormalCompletion}).\par
Opposed to subnormal completions, the case $ \kappa = \infty $ is trivial. Since by \cite[Corollary 7.2.3]{JJS1} the only completely hyperexpansive weighted shifts on $ \mathcal{T}_{\eta,\infty} $ are isometries, the existence of completely hyperexpansive completion is guaranteed by the following three equalities
\begin{align*}
    \lambda_{-k} &= 1, \qquad k\in\mathbb{N}\\
    \lambda_{i,j} &= 1, \qquad i\in\mathbb{N}\cap [1,\eta],\ j\in\mathbb{N}_2\\
    \sum_{i=1}^{\eta}\lambda_{i,1}^{2} &= 1.
\end{align*}
	\section*{Declarations}
	\textbf{Conflict of interests} The author declares that he has no conflict of interests.
	\bibliography{references}

\begin{thebibliography}{20}
\providecommand{\natexlab}[1]{#1}
\providecommand{\url}[1]{\texttt{#1}}
\expandafter\ifx\csname urlstyle\endcsname\relax
  \providecommand{\doi}[1]{doi: #1}\else
  \providecommand{\doi}{doi: \begingroup \urlstyle{rm}\Url}\fi

\bibitem[Athavale(1996)]{Ath1}
Ameer Athavale.
\newblock On completely hyperexpansive operators.
\newblock \emph{Proc. Amer. Math. Soc.}, 124\penalty0 (12):\penalty0
  3745--3752, 1996.
\newblock \doi{10.1090/S0002-9939-96-03609-X}.

\bibitem[Billingsley(1968)]{B1}
Patrick Billingsley.
\newblock \emph{Convergence of Probability Measures}.
\newblock {John Wiley \& Sons, Inc., New York-London-Sydney}, 1968.

\bibitem[Curto and Fialkow(1993)]{CF1}
Ra{\'u}l~E. Curto and Lawrence~A. Fialkow.
\newblock Recursively generated weighted shifts and the subnormal completion
  problem.
\newblock \emph{Integr equ oper theory}, 17\penalty0 (2):\penalty0 202--246,
  June 1993.
\newblock \doi{10.1007/BF01200218}.

\bibitem[Curto et~al.(2010)Curto, Lee, and Yoon]{CLY1}
Ra{\'u}l~E. Curto, Sang~Hoon Lee, and Jasang Yoon.
\newblock A new approach to the 2-variable {{Subnormal Completion Problem}}.
\newblock \emph{Journal of Mathematical Analysis and Applications},
  370\penalty0 (1):\penalty0 270--283, October 2010.
\newblock \doi{10.1016/j.jmaa.2010.04.061}.

\bibitem[Exner et~al.(2018)Exner, Jung, Stochel, and Yun]{EJSY1}
George~R. Exner, Il~Bong Jung, Jan Stochel, and Hye~Yeong Yun.
\newblock A {{Subnormal Completion Problem}} for {{Weighted Shifts}} on
  {{Directed Trees}}.
\newblock \emph{Integr. Equ. Oper. Theory}, 90\penalty0 (6):\penalty0 72,
  December 2018.
\newblock \doi{10.1007/s00020-018-2496-9}.

\bibitem[Exner et~al.(2020)Exner, Jung, Stochel, and Yun]{EJSY2}
George~R. Exner, Il~Bong Jung, Jan Stochel, and Hye~Yeong Yun.
\newblock A {{Subnormal Completion Problem}} for {{Weighted Shifts}} on
  {{Directed Trees}}, {{II}}.
\newblock \emph{Integr. Equ. Oper. Theory}, 92\penalty0 (1):\penalty0 8,
  February 2020.
\newblock \doi{10.1007/s00020-020-2565-8}.

\bibitem[Gellar and Wallen(1970)]{GelWal1}
R.~Gellar and L.~J. Wallen.
\newblock Subnormal weighted shifts and the {H}almos-{B}ram criterion.
\newblock \emph{Proc. Japan Acad.}, 46:\penalty0 375--378, 1970.

\bibitem[Halmos(1970)]{Hal2}
P.~R. Halmos.
\newblock Ten problems in {{Hilbert}} space.
\newblock \emph{Bull. Amer. Math. Soc.}, 76\penalty0 (5):\penalty0 887--934,
  September 1970.
\newblock \doi{10.1090/S0002-9904-1970-12502-2}.

\bibitem[Jab{\l}o{\'n}ski et~al.(2012)Jab{\l}o{\'n}ski, Jung, and
  Stochel]{JJS1}
Zenon Jab{\l}o{\'n}ski, Il~Bong Jung, and Jan Stochel.
\newblock Weighted shifts on directed trees.
\newblock \emph{Memoirs of the AMS}, 216\penalty0 (1017):\penalty0 0--0, 2012.
\newblock \doi{10.1090/S0065-9266-2011-00644-1}.

\bibitem[Jab{\l}o{\'n}ski et~al.(2011)Jab{\l}o{\'n}ski, Jung, Kwak, and
  Stochel]{JabStoJungKwak1}
Zenon~J. Jab{\l}o{\'n}ski, Il~Bong Jung, Jung~Ah Kwak, and Jan Stochel.
\newblock Hyperexpansive completion problem via alternating sequences; an
  application to subnormality.
\newblock \emph{Linear Algebra and its Applications}, 434\penalty0
  (12):\penalty0 2497--2526, June 2011.
\newblock \doi{10.1016/j.laa.2011.01.004}.

\bibitem[Karlin and Studden(1966)]{KS1}
Samuel Karlin and William~J. Studden.
\newblock \emph{Tchebycheff Systems: {W}ith Applications in Analysis and
  Statistics}, volume Vol. XV.
\newblock {Interscience Publishers John Wiley \textbackslash\& Sons, New
  York-London-Sydney}, 1966.

\bibitem[Kimsey(2016)]{K1}
David~P. Kimsey.
\newblock The subnormal completion problem in several variables.
\newblock \emph{Journal of Mathematical Analysis and Applications},
  434\penalty0 (2):\penalty0 1504--1532, February 2016.
\newblock \doi{10.1016/j.jmaa.2015.09.063}.

\bibitem[Kre{\u \i}n and Nudel'man(1977)]{KN1}
M.~G. Kre{\u \i}n and Adol'f~Abramovich Nudel'man.
\newblock \emph{{The Markov moment problem and extremal problems: ideas and
  problems of P. L. \v{C}eby\v{s}ev and A. A. Markov and their further
  development}}.
\newblock Number v. 50 in {Translations of mathematical monographs}. {American
  Mathematical Society}, {Providence, R.I}, 1977.
\newblock ISBN 978-0-8218-4500-4.

\bibitem[Lee(2023)]{Lee1}
Eun~Young Lee.
\newblock A completely hyperexpansive completion problem for weighted shifts on
  directed trees with one branching vertex.
\newblock \emph{Filomat}, 37\penalty0 (12):\penalty0 4047--4064, 2023.

\bibitem[Li(2003)]{Li1}
Chunji Li.
\newblock Two variable subnormal completion problem.
\newblock \emph{Hokkaido Math. J.}, 32\penalty0 (1), February 2003.
\newblock \doi{10.14492/hokmj/1350652422}.

\bibitem[Rogosinski(1958)]{R1}
W.~W. Rogosinski.
\newblock Moments of non-negative mass.
\newblock \emph{Proc. R. Soc. Lond. A}, 245\penalty0 (1240):\penalty0 1--27,
  May 1958.
\newblock \doi{10.1098/rspa.1958.0062}.

\bibitem[Stampfli(1966)]{Sta1}
J.~G. Stampfli.
\newblock Which weighted shifts are subnormal?
\newblock \emph{Pacific J. Math.}, 17:\penalty0 367--379, 1966.

\bibitem[Szafraniec(1992)]{Sza1}
Franciszek~Hugon Szafraniec.
\newblock On extending backwards positive definite sequences.
\newblock \emph{Numer Algor}, 3\penalty0 (1):\penalty0 419--425, December 1992.
\newblock \doi{10.1007/BF02141949}.

\bibitem[Weisstein()]{W1}
Eric~W. Weisstein.
\newblock Vieta's formulas. {From MathWorld---A Wolfram Web Resource}.
\newblock \url{https://mathworld.wolfram.com/VietasFormulas.html}.

\bibitem[Wright(1956)]{Wri1}
Fred~M. Wright.
\newblock On the backward extension of positive definite {{Hamburger}} moment
  sequences.
\newblock \emph{Proc. Amer. Math. Soc.}, 7\penalty0 (3):\penalty0 413--422,
  1956.
\newblock \doi{10.1090/S0002-9939-1956-0080175-0}.

\end{thebibliography}

\end{document}